\documentclass[11pt,a4paper,twoside,reqno]{amsart}

\title[Ricci flow, Killing spinors, and T-duality]{Ricci flow, Killing spinors, and T-duality in generalized geometry}
\author[M. Garcia-Fernandez]{Mario Garcia-Fernandez}
\address{Dep. Matem\'aticas, Universidad Aut\'onoma de Madrid, and Instituto de Ciencias Matem\'aticas (CSIC-UAM-UC3M-UCM), Cantoblanco, 28049 Madrid, Spain}
\email{mario.garcia@icmat.es}

\thanks{This work is funded by the European Union's Horizon 2020 research and innovation programme under the Marie Sklodowska-Curie grant agreement No 655162.}

\usepackage{hyperref,url}
\usepackage{amssymb,latexsym}
\usepackage{amsmath,amsthm}
\usepackage[latin1]{inputenc}
\usepackage{enumerate}
\usepackage[all]{xy} \CompileMatrices
\usepackage{slashed} 
\usepackage{youngtab}
\usepackage{tensor}

\SelectTips{cm}{12} 
\theoremstyle{plain}
\newtheorem{theorem}{Theorem}[section]
\newtheorem{lemma}[theorem]{Lemma}
\newtheorem{corollary}[theorem]{Corollary}
\newtheorem{proposition}[theorem]{Proposition}

\theoremstyle{definition}
\newtheorem{definition}[theorem]{Definition}
\newtheorem{definition-theorem}[theorem]{Definition-Theorem}
\newtheorem{example}[theorem]{Example}
\theoremstyle{remark}
\newtheorem{remark}[theorem]{Remark}

\numberwithin{equation}{section} 

\newcommand{\la}{\langle} \newcommand{\ra}{\rangle}

\newcommand{\tr}{\operatorname{tr}}
\newcommand{\Id}{\operatorname{Id}}

\newcommand{\End}{\operatorname{End}}

\newcommand{\Ker}{\operatorname{Ker}}
\newcommand{\ad}{\operatorname{ad}}

\newcommand{\RR}{{\mathbb R}}

\renewcommand{\(}{\left(}
\renewcommand{\)}{\right)}

\newcommand{\surj}{\to\kern-1.8ex\to}

\newcommand{\cD}{\mathcal{D}}

\begin{document}

\maketitle

\begin{center}
{\small \emph{Dedicated to Nigel Hitchin on the occasion of his seventieth birthday.}}
\end{center}

\begin{abstract}%
We introduce a notion of Ricci flow in generalized geometry, extending a previous definition by Gualtieri on exact Courant algebroids. Special stationary points of the flow are given by solutions to first-order differential equations, the Killing spinor equations, which encompass special holonomy metrics with solutions of the Hull-Strominger system. Our main result investigates a method to produce new solutions of the Ricci flow and the Killing spinor equations. For this, we consider T-duality between possibly topologically distinct torus bundles endowed with 
Courant structures, and demonstrate that 
solutions of the equations are exchanged under this symmetry. As applications, we give a mathematical explanation of the \emph{dilaton shift} in string theory and prove that the Hull-Strominger system is preserved by T-duality.
\end{abstract}


\section{Introduction}

In generalized geometry, as initiated by Hitchin \cite{Hit1}, the condition of having zero torsion does not determine a metric connection uniquely. Consequently, curvature quantities are often difficult to tackle. In this work we study various curvature quantities in generalized geometry, along with natural partial differential equations for a generalized metric. Given a solution of the equations with a large isometry group, our prime motivation is an attempt to produce another solution in a manifold with possibly different topology. 

We begin by introducing a notion of Ricci flow in an arbitrary Courant algebroid $E$, which extends Gualtieri's definition of the generalized Ricci flow on exact Courant algebroids, studied in \cite{He,Streets}. One novelty of our definition is that the gauge-fixed versions of the flow are part of the geometric framework, via the family of divergence operators on the Courant algebroid (see Definition \ref{def:divergenceop}). When $E$ admits a spinor bundle, the gauge-fixing condition 
can be recast more intrinsically in terms of a twisted version of the Dirac generating operators introduced by Alekseev-Xu \cite{AXu}, and studied further by \v Severa \cite{Severa}.  

The relationship between the Ricci flow and the Killing spinor equations in generalized geometry \cite{grt}--for which we give a general definition-- is also studied, showing that the Killing spinors are special stationary points of the flow given by first-order differential equations (see Proposition \ref{prop:Ricciflat}). Solutions of the Killing spinor equations \eqref{eq:killing} are very interesting geometric objects, which encompass special holonomy metrics with solutions of the Hull-Strominger system of partial differential equations \cite{HullTurin,Strom}.

We next investigate a method to produce new solutions of the Ricci flow and the Killing spinor equations 
based on T-duality, a relation between quantum field theories discovered by physicists. The idea has its origins in the literature on string theory, namely the work of Buscher \cite{Buscher1}, which was developed further by Ro\v cek and Verlinde \cite{RoVer}. 
The method is to start with a given solution with abelian symmetries, reduce the space, and 
then produce another solution with dual symmetries. 
Our main result relies on an important observation by Cavalcanti and Gualtieri \cite{CaGu}, which states that T-duality can be viewed as an isomorphism between 
Courant structures on two possibly topologically distinct manifolds.

To state our main result, let us introduce some notation. Let $E$ be a Courant algebroid over a principal $T$-bundle $M$, with $T$ a torus, such that the $T$-action on $M$ lifts to $E$ preserving the Courant algebroid structure. We shall call $(E,M,T)$ an equivariant Courant algebroid. The simple reduction of $E$ by $T$, which we denote by
$$
E/T \to B,
$$
is a Courant algebroid over $B$, whose sheaf of sections is given by the invariant sections of $E$. Given now a pair of equivariant Courant algebroids $(E,M,T)$ and $(\hat E, \hat M, \hat T)$ over a common base
\begin{equation*}
  \xymatrix{
 M \ar[rd]_{p} &  & \hat{M} \ar[ld]^{\hat{p}} \\
  & B & \\
  }
\end{equation*}
we say they are dual if there exists an isomorphism of Courant algebroids between the simple reductions
$$
\psi \colon E/T \to \hat E/\hat T,
$$
which we call the duality isomorphism. We state now in an informal way our main result, and refer to Theorem \ref{th:duality} for the details.

\begin{theorem}\label{th:Tdualityintro}
Invariant solutions of the Ricci flow and the Killing spinor equations are exchanged under the duality isomorphism.
\end{theorem}

The definition of T-duality which we consider here is a straightforward generalization of the main implication of 
\cite[Theorem 3.1]{CaGu}, which states that topological T-duality for principal torus bundles--as defined by Bouwknegt, Evslin and Mathai \cite{BEM}--induces a duality isomorphism. 
In addition to topological T-duality, it comprises as a particular case the heterotic T-duality for transitive Courant algebroids in \cite{BarHek}. 
We should mention that our proof works also when $T$ and $\hat T$ are substituted by arbitrary non-abelian groups, but we have not been able to find any interesting examples in this case. Based on our proof, we believe that Theorem \ref{th:Tdualityintro} extends to a fairly general class of Poisson-Lie T-duals, in the sense of Klim\v c\'ik and \v Severa \cite{KlSevera}. We thank \v Severa for clarifications about the non-abelian setup.


One difficulty we have faced in the proof of Theorem \ref{th:Tdualityintro} is that the curvature quantities which are primarily attached to a generalized metric $V_+ \subset E$ are not an invariant of $V_+$, but rather depend on the choice of a metric-compatible torsion-free generalized connection $D$ on $E$ (which always exists, by Proposition \ref{prop:existence}). Thus, the first objective of this work is to define curvature quantities which only depend on the generalized metric $V_+$ upon a choice of divergence operator 
$$
div \colon \Gamma(E) \to C^\infty(M).
$$
A detailed analysis of this question leads us to the definition of the Ricci tensor for a pair $(V_+,div)$ which is independent of choices (see Definition \ref{def:Ricci}). The relation with the Killing spinor equations is established in Proposition \ref{prop:Ricciflat} via a spinorial formula for the Ricci tensor \eqref{eq:Ricci+-op}, that relies in an algebraic Bianchi identity \eqref{eq:bianchi} for the curvature of a torsion-free generalized connection. This is a remarkable property, as the construction of such objects 
typically involves standard connections with skew-symmetric torsion in the tangent bundle of $M$. Formula \eqref{eq:Ricci+-op} can be taken as an alternative definition of the Ricci tensor, without relying on Proposition \ref{propo:Riccitorsion}.


As an application of our framework we revisit some aspects of topological T-duality for exact Courant algebroids. 
The novelty here is to understand the notion of duality for pairs $(V_+,div)$ in this particular context (see Definition \ref{def:dualpairs}), with the upshot of a mathematical explanation of the \emph{dilaton shift} in string theory, as originally observed by Buscher \cite{Buscher1} (see Remark \ref{rem:dilaton}). It would be interesting to compare our general formula \eqref{eq:dilatonshift} for the dilaton shift with \cite[Eq. (3.16)]{OssaQue} in the context of non-abelian T-duality in physics.

The last part of this work is devoted to prove--in Theorem \ref{th:Strduality}--that the solutions of the Hull-Strominger system of partial differential equations are preserved by heterotic T-duality \cite{BarHek}. These equations have its origins in supergravity in physics \cite{HullTurin,Strom}, and they were first considered in the mathematics literature in a seminal paper by Li and Yau \cite{LiYau} (see \cite{GF2} for a review). 
Our proof has two essential ingredients: firstly, the characterization of the Hull-Strominger system in terms of the Killing spinor equations in generalized geometry in \cite{grt} combined with Theorem \ref{th:Tdualityintro}, and, secondly, a general formula for the dilaton shift under T-duality in Proposition \ref{prop:dilatonshift}.
Our result in Theorem \ref{th:Strduality} builds towards the definition of a Strominger-Yau-Zaslow version of mirror symmetry for the Hull-Strominger system, as proposed in \cite{Yau2005}. 
Previous related work, in the realm of semi-flat SYZ mirror symmetry for non-K\"ahler Calabi-Yau manifolds without bundles, can be found in \cite{LTY}.

It is interesting to notice that Proposition \ref{prop:specialholonomy} combined with Theorem \ref{th:duality} implies that metrics with parallel spinors and a continuous abelian group of isometries are preserved by T-duality (cf. \cite{H2,SYZ}). In this way, we recover the observation of \cite{H2,Leung,LYZ,SYZ} that Calabi-Yau, $G_2$ and $Spin(7)$-metrics with torus symmetries are preserved by dualisation of the fibres, with a new proof using spinors which is independent of the dimension. Based on the details of our proof, we believe that this is also true for the non-abelian Poisson-Lie T-duality \cite{KlSevera}. Despite the fact that compact manifolds with special holonomy have no continuous symmetries, it would be worth exploring this perspective of the present work in the abundant local examples that exist in the literature (see e.g. \cite{Salamon}).

The constructions presented in this paper work in arbitrary Courant algebroids, and in particular lead to notions of Ricci flow and Killing spinors for `Courant algebroids over a point', that is, for quadratic Lie algebras. I hope to go back to this interesting interplay with quantum algebra in future work. 



Since this paper was finished, there have been several developments in generalized geometry which build on the results and methods presented in this work \cite{Jurco,SV1}. Furthermore, our Theorem \ref{th:Strduality} has been recently used in \cite{GF2018} to find the first examples of T-dual solutions of the Hull-Strominger system on compact non-K\"ahler manifolds with different topology.

\vspace{12pt}

\noindent
{\bf Acknowledgements.}
I am grateful to Nigel Hitchin for stimulating discussions and insight. The present work is motivated by his teaching during the lecture course `Topics in Generalized Geometry', held at the Mathematical Institute in Oxford in 2012. I am indebted to Roberto Rubio, with whom I did some of the calculations in Proposition \ref{prop:weylfixed}, and for his careful reading of the manuscript. I wish also to thank Luis \'Alvarez-C\'onsul, Xenia de la Ossa, Vicente Cort\'es, Marco Gualtieri, Yolanda Lozano, Pavol \v Severa, Jeffrey Streets, and Daniel Waldram for helpful discussions and comments about the manuscript.


\section{Connections and Dirac generating operators}\label{sec:Dirac}

\subsection{Torsion and divergence of generalized connections}\label{subsec:connection}

A \emph{Courant algebroid} over a manifold $M$ is a vector bundle $E \to M$ equipped with a
fibrewise nondegenerate symmetric bilinear form $\la\cdot,\cdot\ra$, a bilinear bracket $[\cdot,\cdot]$ on the smooth sections $\Gamma(E)$, and a bundle map $\pi \colon E \to TM$ called the anchor, which satisfy the following conditions for all $e_1, e_2, e_3 \in \Gamma(E)$ and $f \in C^\infty(M)$
\begin{itemize}
  \item[(C1):] $[e_1,[e_2,e_3]] = [[e_1,e_2],e_3] + [e_2,[e_1,e_3]]$,
  \item[(C2):] $\pi([e_1,e_2])=[\pi(e_1),\pi(e_2)]$,
  \item[(C3):] $[e_1,f e_2] = \pi(e_1)(f) e_2 + f[e_1,e_2]$,
  \item[(C4):] $\pi(e_1)\la e_2, e_3 \ra = \la [e_1,e_2], e_3 \ra + \la e_2,
    [e_1,e_3] \ra$,
  \item[(C5):] $2[e_1,e_1]=\pi^* d\la e_1,e_1\ra$,
  \end{itemize}
where in the last equality we use $\la\cdot,\cdot\ra$ to identify $E$ with $E^*$. We note that condition (C2) is redundant, as it follows from (C1) and (C3).

A \emph{generalized connection} on the Courant algebroid $E$ is an $E$-connection compatible with the inner product \cite{AXu,G3}, that is, a first-order differential operator
$$
D \colon \Gamma(E) \to \Gamma(E^* \otimes E)
$$
satisfying the Leibniz rule and compatible with $\la\cdot,\cdot\ra$:
\begin{equation}\label{eq:generalizedconn}
\begin{split}
D_{e_1}(fe_2) &= \pi(e_1)(f)e_2 + fD_{e_1}e_2,\\
\pi(e_1)\langle e_2,e_3 \rangle & = \langle D_{e_1} e_2,e_3 \rangle + \langle e_2,D_{e_1} e_3 \rangle.
\end{split}
\end{equation}
We will denote $D_{e_1}(e_2) = \iota_{e_1} D(e_2)$, where $\iota$ is the contraction operator of elements in $E^*$ with elements in $E$.

The set of generalized connections on $E$, that we will denote by $\cD$, has a structure of affine space modelled on the vector space
$$
\Gamma(E^* \otimes \mathfrak{o}(E)),
$$
where $\mathfrak{o}(E)$ denotes the bundle of skew-symmetric endomorphisms of $E$ with respect to the bilinear form $\langle\cdot,\cdot\rangle$.
To see this, given a (standard) orthogonal connection $\nabla^E$ on $(E,\langle\cdot,\cdot\rangle)$ we can construct a generalized connection on $E$
\begin{equation}\label{eq:generalizedconnexp}
D_{e_1} e_2 = \nabla^E_{\pi(e_1)}e_2,
\end{equation}
so that any other generalized connection on $E$ is of the form $D + \chi$, for an element $\chi \in \Gamma(E^* \otimes \mathfrak{o}(E))$.

\begin{definition}\label{def:torsion}
The torsion $T_D \in \Gamma(\Lambda^3 E^*)$ of a generalized connection $D$ on $E$ is defined by
\begin{equation}\label{eq:torsionG}
T_D(e_1,e_2,e_3) = \la D_{e_1}e_2 - D_{e_2}e_1 - [e_1,e_2],e_3 \ra + \la D_{e_3}e_1, e_2 \ra.
\end{equation}
\end{definition}

Using the skew-symmetrized (Courant) bracket $[[e_1,e_2]] = \frac{1}{2}([e_1,e_2] - [e_2,e_1])$ and the compatibility of $D$ with the bilinear form on $E$, it is immediate to see that this definition is equivalent to Gualtieri's torsion, as defined in \cite{G3}. A less direct (although elementary) calculation shows that Definition \ref{def:torsion} is equivalent to Alexseev-Xu's \cite{AXu}, given by the formula
\begin{equation}\label{eq:torsionAXu}
T_D(e_1,e_2,e_3) = c.p. \( \frac{1}{2}\la D_{e_1}e_2 - D_{e_2}e_1, e_3 \ra  - \frac{1}{3} \la [[e_1,e_2]],e_3 \ra \),
\end{equation}
where $c.p.$ stands for sum over cyclic permutations. We note that formula \eqref{eq:torsionAXu} is more general than \eqref{eq:torsionG}, as it defines an element in $\Gamma(\Lambda^3 E^*)$ also for $E$-connections which are not compatible with the ambient metric.

Let $T$ be an element in $\Gamma(\Lambda^3 E^*)$, and consider the set $\cD^T$ of generalized connections on $E$ with fixed torsion $T$
$$
\cD^T \subset \cD.
$$
For a choice of $D \in \cD^T$ and $\chi \in \Gamma(E^* \otimes \mathfrak{o}(E))$, the condition for $D' = D + \chi$ to be in $\cD^T$ is given by
\begin{equation}\label{eq:cyclic}
c.p. \la \chi_{e_1}e_2,e_3\ra = 0, 
\end{equation}
where the left hand side corresponds to the total skew-symmetrization of $\la \chi_{e_1}e_2,e_3\ra$. Using the canonical identification $E \otimes \Lambda^2 E \cong E^* \otimes \mathfrak{o}(E)$, we can regard the affine space $\cD^T$ as modelled on the following vector space of mixed symmetric $3$-tensors
\begin{equation}\label{eq:sigma}
\Sigma = \{\sigma \in \Gamma(E^{\otimes 3}): \sigma_{123} = - \sigma_{132}, \; c.p. \;\sigma_{123} = 0\},
\end{equation}
where $\sigma_{123} = \sigma(e_1,e_2,e_3)$ for $e_i \in \Gamma(E)$ (using $\la \cdot, \cdot\ra$ to identify $E \cong E^*$). Equivalently, considering the canonical exact sequence
\begin{equation}
\label{eq:exacttorsion}
\begin{gathered}
  \xymatrix{ & 0 \ar[r] & S^3 E \ar[r] & S^2 E \otimes E \ar[r] & E \otimes \Lambda^2 E \ar[r] & \Lambda^3 E \ar[r] & 0,}
\end{gathered}
\end{equation}
where $S^kE$ denotes the $k$-th symmetric product, we have 
$$
\Sigma \cong \Gamma((S^2 E \otimes E)/S^3 E).
$$
From the previous sequence we observe that there is a distinguished subspace
$$
E \subset S^2 E \otimes E \colon e \to \la\cdot,\cdot\ra^{-1} \otimes e
$$
inducing a direct sum decomposition
$$
S^2 E \otimes E = (S^2 E \otimes E)_0 \oplus E,
$$
where $(S^2 E \otimes E)_0$ is the subbundle generated by elements of the form
$$
(S^2 E \otimes E)_0 = \la (e_1 \odot e_2) \otimes e: (e_1,e_2) = 0\ra. 
$$
From this, we obtain a splitting
\begin{equation}\label{eq:splittingyoung}
\Sigma = \Sigma_0 \oplus \Gamma(E)
\end{equation}
where $e \in \Gamma(E)$ corresponds to the mixed symmetric tensor $\sigma^e$ defined by
$$
\sigma^e(e_1,e_2,e_3) = \la e_1,e_2 \ra \la e,e_3\ra - \la e,e_2 \ra \la e_1,e_3 \ra
$$
and the complement of $\Gamma(E)$ is given by
\begin{equation}\label{eq:sigma0}
\Sigma_0 = \Big{\{}\sigma \in \Sigma : \sum_{i=1}^{r_E} \sigma(e_i,\tilde e_i,\cdot) = 0\Big{\}}.
\end{equation}
Here, $r_E$ denotes the rank of $E$, $\{e_i\}$ is an orthogonal local frame for $E$ and the $\{\tilde e_i\}$ are sections of $E$ defined so that $\la e_i, \tilde e_j \ra = \delta_{ij}$. More explicitly, given an arbitrary element $\chi \in \Gamma(E^* \otimes \mathfrak{o}(E))$ we have a unique decomposition
\begin{equation}\label{eq:decompositionchi}
\chi = \chi_0 + \chi^e,
\end{equation}
where $\chi_0 \in \Gamma(E^* \otimes \mathfrak{o}(E))$ is such that
$$
\sum_{i=1}^{r_E} (\chi_0)_{e_i}\tilde e_i = 0, 
$$
and
$$
\chi^e_{e_1}e_2 = \la e_1,e_2 \ra e - \la e,e_2 \ra e_1
$$
with 
$$
e = \frac{1}{r_E -1}\sum_{i=1}^{r_E} \chi_{e_i}\tilde e_i.
$$

As observed in \cite{GF}, the $E^*$-valued skew-symmetric endomorphism $\chi^e$ in the decomposition \eqref{eq:decompositionchi} is reminiscent of the `$1$-form valued Weyl endomorphisms' in pseudo-Riemannian geometry, which appear in the variation of a metric connection with fixed torsion upon a conformal change of the metric. Similarly, our ($E^*$-valued) \emph{Weyl endomorphisms} $\chi^e$ enable us to deform a generalized connection $D$ with fixed torsion $T$ inside the space $\cD^T$. Notice that the notion of generalized connection--as considered here--does not allow for conformal changes of the ambient metric $\la\cdot,\cdot\ra$. 

For later applications in this work, it is important to `gauge-fix' the Weyl degrees of freedom in the space of generalized connections $\cD^T$ corresponding to $\Gamma(E)$ in \eqref{eq:splittingyoung}. We will call this procedure  \emph{Weyl gauge fixing}. First, we recall the notion of \emph{divergence} of a generalized connection introduced in \cite{AXu}. Given a generalized connection $D$ on $E$ and $e \in \Gamma(E)$, we can define an endomorphism $De$ of $E$ by $e' \to D_{e'}e$.

\begin{definition}[\cite{AXu}]\label{def:divergence}
The \emph{divergence} of $e \in \Gamma(E)$ with respect to a generalized connection $D$ is
\begin{equation}\label{eq:divergence}
div_D(e) = \tr De \in C^\infty(M).
\end{equation}
\end{definition}

The divergence of a generalized connection defines a first-order differential operator
$$
div_D\colon \Gamma(E) \to C^\infty(M)
$$
satisfying the $\pi$-Leibniz rule
\begin{equation}\label{eq:Leibnizdiv}
div_D(f e) =  \pi(e)(f) + f div_D e,
\end{equation}
for any $f \in C^\infty(M)$. For our applications, it will be useful to have a definition independent of the choice of a generalized connection, which we shall give now.

\begin{definition}\label{def:divergenceop}
A differential operator $div \colon \Gamma(E) \to C^\infty(M)$ satisfying \eqref{eq:Leibnizdiv} will be called a \emph{divergence operator} on $E$.
\end{definition}

Note that the divergence operators on $E$ form an affine space modelled on $\Gamma(E^*)$. Our next result shows that the divergence constrains the Weyl degrees of freedom in the space generalized connections with fixed torsion, as required.

\begin{lemma}\label{lem:weylfixed}
Let $T \in \Gamma(\Lambda^3E^*)$ and $div \colon \Gamma(E) \to C^\infty(M)$ be a divergence operator on $E$. Then
$$
\cD^T(div) = \{D \in \cD \; | \; T_D = T, \; div_D = div \} \subset \cD^T
$$
is an affine space modelled on $\Sigma_0$, as defined in \eqref{eq:sigma0}.
\end{lemma}

\begin{proof}
Given $D', D \in \cD^T$, we denote $D' = D + \chi$ with $\chi \in \Sigma$ (see \eqref{eq:sigma}). Then, by definition of the divergence of $e' \in \Gamma(E)$ we have
\begin{equation}\label{eq:divergencevariation}
div_{D'}(e') = div_{D}(e) - \sum_{i=1}^{r_E}\la \chi_{\tilde e_i} e_i,e'\ra.
\end{equation}
Decomposing $\chi = \chi_0 + \chi^e$ as in \eqref{eq:decompositionchi}, for $e \in \Gamma(E)$, we obtain
$$
\sum_{i=1}^{r_E}\chi_{\tilde e_i} e_i = (r_E - 1)e.
$$
Imposing now $div_{D'} = div_D = div$, implies $e = 0$, which concludes the proof.
\end{proof}

\begin{example}\label{ex:divergence}
Let $E = T \oplus T^*$ be an exact Courant algebroid over $M$ with pairing
$$
\la X + \xi, X + \xi \ra = \xi(X),
$$
anchor $\pi(X+\xi) = X$, and bracket
$$
[X+\xi, Y+\eta] = [X,Y] + L_X \eta - \iota_Y d\xi + \iota_Y \iota_X H,
$$
for $H \in \Gamma(\Lambda^3 T^*)$ a closed $3$-form. Given a connection $\nabla$ on $T$, we consider the generalized connection $D$ on $E$ induced by the orthogonal connection $\nabla \oplus \nabla^*$ as in \eqref{eq:generalizedconnexp},
where $\nabla^*$ denotes the induced connection on the cotangent bundle $T^*$. Then
\begin{equation}\label{eq:torsionexample}
T_D(e_1,e_2,e_3) = - \frac{1}{2}\pi^*H(e_1,e_2,e_3) + c.p. \la s(\pi^*T_\nabla(e_1,e_2)),e_3\ra,
\end{equation}
where $s \colon T \to E$ denotes the isotropic splitting $s(X) = X$ and $T_\nabla$ is the torsion of the connection $\nabla$, and
$$
div_D(e) = \tr \nabla X,
$$
for $e = X + \xi$. Assuming now $T_\nabla = 0$, so that $\tr \nabla X = \tr \nabla_X - L_X$, and that there exists a density $\mu \in \Gamma(|\det T^*|)$ preserved by $\nabla$, we obtain
$$
div_D(e) \mu = L_X \mu,
$$
which recovers the standard notion of divergence for the vector field $\pi(e) = X$.
\end{example}

\subsection{Dirac generating operators and Weyl gauge fixing}\label{subsec:genoper}

The aim of this section is to provide a different point of view on the Weyl gauge fixing condition in Lemma \ref{lem:weylfixed}.
In order to do this, we build on the spinorial geometry of $E$ and the theory of \emph{Dirac generating operators} in \cite{AXu,GMXu,Severa}. We hope that our discussion motivates the use of torsion-free generalized connections in the rest of this work.

Let $Cl(E)$ denote the bundle of Clifford algebras of the orthogonal bundle  $(E,\la\cdot,\cdot\ra)$, defined by the relation
$$
e^2 = \la e,e \ra
$$
for any section $e \in \Gamma(E)$. In the rest of this section, we assume that $Cl(E)$ admits an irreducible Clifford module $S$, so that the Clifford action of sections of $Cl(E)$ on sections of $S$ induces an isomorphism $\Gamma(\End S) \cong \Gamma(Cl(E))$. We will assume that the \emph{spinor bundle} $S$ is real, but our discussion can be carried out in an analogue way in the complex case. 

The $\mathbb{Z}_2$-grading on $\Gamma(Cl(E))$ induces a $\mathbb{Z}_2$-grading on the differential operators on $\Gamma(S)$, such that the functions $f \in C^\infty(M)$ acting by multiplication on $\Gamma(S)$ are even operators. Consider the graded commutator on the space of differential operators on $\Gamma(S)$, given by
$$
[L_0,L_1] = L_1L_2 - (-1)^{\deg L_1 \deg L_2}L_2 L_1.
$$
\begin{definition}[\cite{AXu}]\label{def:genoper}
A first-order odd differential operator $\slashed d$ on $\Gamma(S)$ is called a \emph{Dirac generating operator} of the Courant algebroid $E$ if it satisfies the following conditions for every $f \in C^\infty(M)$ and $e_1,e_2 \in \Gamma(E)$
\begin{enumerate}[i)]
\item $[[\slashed d, f],e\cdot] = \pi(e)(f)$,

\item $[[\slashed d,e_1\cdot ],e_2 \cdot ] = [e_1,e_2] \cdot$,

\item $\slashed d^2 \in C^\infty(M)$.

\end{enumerate}
\end{definition}

According to the previous definition, a Dirac generating operator indeed generates the Courant algebroid structure--that is, the bracket and the anchor map-- from the orthogonal bundle $(E,\la \cdot,\cdot \ra)$. The existence of a Dirac generating operator (or simply a generating operator) for a given Courant algebroid which admits a spinor bundle is nontrivial and was one of the main results of \cite{AXu}. Furthermore, in \cite{GMXu,Severa} it was proved that, provided that it exists, there is a canonical choice compatible with \emph{Weyl quatization}. In the following theorem, we summarize the result of this theory that we will need. Let us denote by $|\det T^*|^{1/2}$ the trivializable line bundle of half-densities on $M$. We denote by $r_S$ the rank of $S$.

\begin{theorem}[\cite{AXu,GMXu}]\label{th:genoper}
Assume that the line bundle $(\det S^*)^{\frac{1}{r_S}}$ exists, and consider the twisted spinor bundle
\begin{equation}\label{eq:spinortwistedhalf}
\mathbb{S} = S \otimes (\det S^*)^{\frac{1}{r_S}} \otimes |\det T^*|^{\frac{1}{2}}.
\end{equation}
Then, there exists a canonical generating operator $\slashed d_0$ on $\Gamma(\mathbb{S})$ of $E$. Furthermore, any other generating operator $\slashed d$ is of the form $\slashed d = \slashed d_0  + e \cdot$, for $e\in \Gamma(E)$ such that $[\slashed d_0,e \cdot] \in C^\infty(M)$.
\end{theorem}

The main idea to construct a generating operator on $\Gamma(\mathbb{S})$ is to start with a Dirac operator $\slashed D$ on $\Gamma(S)$ for a choice of generalized connection $D$, and then try to render the construction independent of choices. The definition of $\slashed D$ involves a choice of spin $E$-connection, and this ambiguity is fixed by the twist $(\det S^*)^{\frac{1}{r_S}}$. Another ambiguity comes from the torsion of $D$, regarded as an operator on $\Gamma(S)$ by Clifford multiplication. Finally, once the torsion is fixed, there is still an ambiguity coming from the Weyl endomorphisms in decomposition \eqref{eq:splittingyoung}, which is removed by the twist $|\det T^*|^{\frac{1}{2}}$ and the divergence operator. 

In this context, the \emph{Weyl gauge fixing} in Lemma \ref{lem:weylfixed} amounts to choosing a generating operator $\slashed d$ in the space of sections of a different twisted spinor bundle, namely
\begin{equation}\label{eq:spinortwisted}
\mathcal{S} = S \otimes (\det S^*)^{\frac{1}{r_S}},
\end{equation}
and to constructing a space of generalized connections $D$ on $E$ with fixed Dirac operator $\slashed D = \slashed d$. An interesting consequence of this condition is that $D$ is forced to be torsion-free, which motivates the use of these type of generalized connections in the present work. Note here that, since $|\det T^*|^{\frac{1}{2}}$ is trivializable, a generating operator on $\Gamma(\mathcal{S})$ always exists under the assumptions of Theorem \ref{th:genoper}. 
Before going into any details of our construction, let us discuss briefly a basic example.

\begin{example}\label{example:dirac}
Let $E = T \oplus T^*$ be an exact Courant algebroid over $M$, as considered in Example \ref{ex:divergence}. For the choice of Lagrangian $T^* \subset E$, we have a spinor bundle $S = \Lambda^*T^*$, with Clifford action
\begin{equation}\label{eq:Cliffordnoncan}
(X + \xi) \cdot \alpha = \iota_X \alpha + \xi \wedge \alpha,
\end{equation}
where $\alpha \in \Gamma(\Lambda^*T^*)$. Then, $\det S^* \cong |\det T|^{\frac{r_S}{2}}$ and the twisted spinor bundles are $\mathcal{S} = \Lambda^*T^* \otimes |\det T|^{\frac{1}{2}}$ and
$$
\mathbb{S} =  \Lambda^* T^* \otimes ( \det T \otimes \det T^*)^{\frac{1}{2}}\cong \Lambda^* T^*,
$$
with canonical generating operator on $\mathbb{S}$ given by
$$
\slashed d_0 = d - H \wedge.
$$
If $\slashed d$ is another generating operator on $\mathbb{S}$, then $\slashed d = \slashed d_0 + e \cdot$ and $[\slashed d_0,e\cdot] \in C^\infty(M)$. Setting $e = X + \xi$, we have
$$
[\slashed d_0,e\cdot] = L_X + (\iota_X H + d\xi)\wedge,
$$
and therefore $\pi(e) = X = 0$ and $d\xi = 0$. Choosing now a trivialization $\mu^{\frac{1}{2}}_0$ of $|\det T^*|^{\frac{1}{2}}$, any other trivialization is of the form $\mu^{\frac{1}{2}} = e^f \mu^{\frac{1}{2}}_0 \in \Gamma(|\det T^*|^{\frac{1}{2}})$ for a smooth function $f \in C^\infty(M)$, and therefore 
using the canonical generating operator on $\mathbb{S}$ we have a generating operator $\slashed d$ on $\mathcal{S}$ corresponding to $\mu^{\frac{1}{2}}$:
$$
\slashed d = e^{-f}\slashed d_0 e^f = \slashed d_0 = d - (H - df)\wedge.
$$
In conclusion, the generating operators on $\Gamma(\mathbb{S})$ are parameterized by closed $1$-forms, while the generating operators on $\Gamma(\mathcal{S})$ given by pull-back of $\slashed d_0$ 
form an affine space modelled on exact $1$-forms.
\end{example}

Given a generalized connection $D$ on $E$, we have an induced $E$-connection on $Cl(E)$, which we denote also by
$$
D \colon \Gamma(Cl(E)) \to \Gamma(E^* \otimes Cl(E)).
$$
An $E$-connection on $S$ compatible with $D$ is a first differential operator
$$
D^S \colon \Gamma(S) \to \Gamma(E^* \otimes S)
$$
satisfying the Leibniz rule with respect to the anchor map $\pi$ (cf. \eqref{eq:generalizedconn}) and such that
$$
D^S(e \cdot \alpha) = D(e)\cdot \alpha + e \cdot D^S \alpha.
$$
Given an $E$-connection on $S$, there is an induced Dirac operator
$$
\slashed D^S \colon \Gamma(S) \to \Gamma(S),
$$
defined explicitly by
$$
\slashed D^S \alpha = \frac{1}{2}\sum_{i=1}^{r_E} \tilde e_i \cdot D^S_{e_i} \alpha
$$
for any orthogonal local frame $\{e_i\}$ of $E$, where $\cdot$ denotes Clifford multiplication. A pair of $E$-connections on $S$ compatible with the same generalized connection $D$ differ by an element $e \in \Gamma(E)$, regarded as the $E^*$-valued endomorphisms of $S$ given by  $\alpha \to \la e,\cdot\ra\otimes \alpha$ (see \cite[Lemma 3.4]{GMXu}). From this, it follows that $D$ induces canonically a Dirac operator on the twisted spinor bundle $\mathcal{S}$ (see \eqref{eq:spinortwisted})
$$
\slashed D \colon \Gamma(\mathcal{S}) \to \Gamma(\mathcal{S}).
$$


\begin{definition}\label{def:weylfixed}
Let $\slashed d \colon \Gamma(\mathcal{S}) \to \Gamma(\mathcal{S})$ be a generating operator of the Courant algebroid $E$. The family of 
generalized connections associated to $\slashed d$ is defined by:
\begin{equation}\label{eq:Diracfamily}
\cD(\slashed d) = \{D \in \cD | \slashed D = \slashed d \}.
\end{equation}
\end{definition}

In the next lemma we prove that $\cD(\slashed d)$ provides a particular example of the family of Weyl-gauge fixed generalized connections in Lemma \ref{lem:weylfixed}.

\begin{proposition}\label{prop:weylfixed}
The space $\cD(\slashed d)$ is non-empty, and all its elements are torsion-free. Furthermore, $\slashed d$ determines uniquely a divergence operator $div$ and, with the notation in Lemma \ref{lem:weylfixed}, there is an equality $\cD(\slashed d) = \cD^0(div)$.
\end{proposition}

\begin{proof}
The proof follows closely the construction of generating Dirac operators in \cite{AXu}. Let $D$ be an arbitrary generalized connection on $E$. Note that the divergence defines an $E$-connection on $\det T^*$ by
\begin{equation}\label{eq:Dlambda}
D^\Lambda_{e_1} \mu = L_{\pi(e_1)}\mu - div_D(e)\mu,
\end{equation}
where $\mu \in \Gamma(\det T^*)$. Choose a trivialization $\mu^{\frac{1}{2}}$ of $|\det T^*|^{\frac{1}{2}}$, inducing an isomorphism $\mathcal{S} \cong \mathbb{S}$, and define $e \in \Gamma(E)$ by
$$
\la e,\cdot \ra = -\frac{1}{2}\mu^{-\frac{1}{2}}D^\Lambda \mu^{\frac{1}{2}} \in \Gamma(E).
$$
Then, by \cite[Proposition 5.12]{GMXu}, the following formula defines the canonical generating operator on $\Gamma(\mathcal{S}) \cong \Gamma(\mathbb{S})$
\begin{equation}\label{eq:genopexp}
\slashed d_0 = \slashed D + \frac{1}{4}T_D \cdot + e \cdot.
\end{equation}
Let $\slashed d = \slashed d_0 + e' \cdot$ be an arbitrary generating operator on $\Gamma(\mathcal{S})$. Using the metric $\la\cdot,\cdot\ra$, we regard the torsion $T_{D}$ as an element in $\Gamma(E^* \otimes \mathfrak{o}(E))$ and consider $D' = D - \frac{1}{3}T_D$, so that $T_D = 0$ (see \eqref{eq:cyclic}) and $\slashed d = \slashed D' + e'' \cdot$, where $e'' = e + e'$. Finally, setting $D'' = D' + \chi^{e'''}$ we have
\begin{equation}\label{eq:diracvariation}
\slashed D'' =  \slashed D' + \frac{r_E - 1}{4}e''' \cdot
\end{equation}
and hence for $e''' = -4e''/(r_E-1)$ we obtain $\slashed d = \slashed D''$.

Assume now that $\slashed d = \slashed D$ for a generating operator $\slashed d$ on $\Gamma(\mathcal{S})$ and a generalized connection $D$ on $E$. Using \emph{Rubio's formula} for the torsion \cite{Rubiopr}, given by
$$
\iota_{e_2} \iota_{e_1} T_D \cdot = [[\slashed D,e_1\cdot],e_2 \cdot] - [e_1,e_2] \cdot
$$
for $e_1,e_2 \in \Gamma(E)$, it follows from property $ii)$ in Definition \ref{def:genoper} that $T_D = 0$. Given $D' = D + \chi$ with $\chi \in \Gamma(E^* \otimes \mathfrak{o}(E))$, we denote
$$
\sigma(e_1,e_2,e_3) = \la \chi_{e_1}e_2,e_3\ra 
$$
and let $c.p.(\sigma)$ be its total skew-symmetrization. Then, a direct calculation shows that
\begin{equation}\label{eq:diracdivergencevariation}
\slashed D' = \slashed D - \frac{1}{4} c.p.(\sigma) \cdot - \frac{1}{4} \sum_{i=1}^{r_E}\chi_{\tilde e_i} e_i \cdot.
\end{equation}
for $e' \in \Gamma(E)$. Decomposing $\chi = \chi_0 + \chi^e$ as in \eqref{eq:decompositionchi}, for $e \in \Gamma(E)$, yields
$$
\sum_{i=1}^{r_E}\chi_{\tilde e_i} e_i = (r_E - 1)e.
$$
Imposing now $\slashed D' = \slashed D = \slashed d$, we obtain $c.p.(\sigma) = 0$ and therefore $e = 0$ by \eqref{eq:diracdivergencevariation}. Finally, \eqref{eq:divergencevariation} implies that $div_{D'}(e') = div_{D}(e')$, which concludes the proof.
\end{proof}


\begin{example}
Consider the exact Courant algebroid $E$ endowed with the generalized connection $D$ in Example \ref{ex:divergence}. Assume that $\nabla$ is torsion-free and that it preserves the density $\mu \in \Gamma(|\det T^*|)$. Then, we have
$$
D^\Lambda_e \mu = \nabla_{\pi(e)}\mu,
$$
and an induced spin connection on $\mathcal{S} \cong \det T^* \otimes |\det T|^{\frac{1}{2}}$ given by
$$
D(\alpha \otimes \mu^{-\frac{1}{2}}) = \nabla (\alpha \otimes \mu^{-\frac{1}{2}}), 
$$
with Dirac operator
$$
\slashed D (\alpha \otimes \mu^{-\frac{1}{2}}) = (d\alpha + (\mu^{\frac{1}{2}} \nabla \mu^{-\frac{1}{2}}) \wedge \alpha) \otimes \mu^{-\frac{1}{2}}.
$$
Combining these formulae with \eqref{eq:torsionexample}, it is easy to see that $\slashed d_0$ coincides with the twisted de Rham differential $d - H \wedge$, as claimed in Example \ref{example:dirac}.
\end{example}

\section{Generalized metrics and torsion-free connections}\label{sec:ggmetrics}

\subsection{The many Levi-Civita connections}\label{sec:ggmetrics}

In this section we investigate an analogue in generalized geometry of the \emph{Fundamental Lemma of pseudo-Riemannian geometry}, that is, the existence of a unique torsion-free connection compatible with a given metric. In generalized geometry there is always a freedom to choose a compatible torsion-free generalized connection, that we will constrain via the Weyl gauge fixing mechanism introduced in Lemma \ref{lem:weylfixed}. This will lead us to the definition of four canonical operators in the next sections, providing a weak analogue of the Fundamental Lemma.

Let $(t,s)$ be the signature of the pairing on the Courant algebroid
$E$. A generalized metric of signature $(p,q)$, or simply a metric on
$E$, is a reduction of the $O(t,s)$-bundle of frames of $E$ to
$$
O(p,q)\times O(t-p,s-q) \subset O(t,s).
$$
Alternatively, it is given by an orthogonal decomposition
$$
E = V_+ \oplus V_-
$$
such that the restriction of the metric on $E$ to $V_+$ is a
non-degenerate metric of signature $(p,q)$. A generalized metric determines
a vector bundle isomorphism
$$
G\colon E \to E,
$$
with $\pm 1$-eigenspace $V_\pm$, which is symmetric, $G^* = G$, and
squares to the identity, $G^2 = \Id$. The endomorphism $G$ determines
completely the metric, as $V_+$ is recovered by
$$
V_+ = \Ker(\Id - G).
$$

\begin{definition}[\cite{G3}]
A generalized connection $D$ on $E$ is compatible with a generalized metric $V_+ \subset E$ if
$$
D (\Gamma(V_\pm)) \subset \Gamma(E^* \otimes V_\pm).
$$
\end{definition}

A generalized connection $D$ compatible with $V_+$ is determined by four first-order differential operators
\begin{align*}
D^+_- & \colon \Gamma(V_+) \to \Gamma(V_-^* \otimes V_+), \qquad D^-_+ \colon \Gamma(V_-) \to \Gamma(V_+^* \otimes V_-),\\
D^+_+ & \colon \Gamma(V_+) \to \Gamma(V_+^* \otimes V_+), \qquad D^-_- \colon \Gamma(V_-) \to \Gamma(V_-^* \otimes V_-),
\end{align*}
satisfying the Leibniz rule (formulated in terms of $\pi_\pm = \pi_{|V_\pm}$). Thus, the space of $V_+$-compatible generalized connections on $E$, which we denote $\cD(V_+)$, is an affine space modelled on
\begin{equation}\label{eq:vectorspDV+}
\Gamma(E^* \otimes \mathfrak{o}(V_+)) \oplus \Gamma(E^* \otimes \mathfrak{o}(V_-)).
\end{equation}
As we shall see next, the \emph{mixed-type operators} $D^\pm_\mp$ are fixed, if we vary a generalized connection $D$ inside $\cD(V_+)$ while preserving the torsion $T_D$. Furthermore, when the torsion is of \emph{pure-type}, that is 
$$
T_D \in  \Gamma(\Lambda^3 V_+^* \oplus \Lambda^3 V_-^*),
$$
then $D^\pm_\mp$ are uniquely determined by the generalized metric and the bracket on $E$. Given $e \in \Gamma(E)$, we denote its orthogonal projection onto $V_\pm$ by
$$
e^\pm = \frac{1}{2}(\Id \pm G)e.
$$

\begin{lemma}\label{lem:mixedfixed}
Let $D \in \cD(V_+)$ with torsion $T \in \Gamma(\Lambda^3 E^*)$. 
\begin{enumerate}[i)]

\item If $D' \in \cD(V_+)$ and $T_{D'} = T$, then 
$(D')^\pm_\mp = D^\pm_\mp$.

\item Furthermore, $T$ is of pure-type if and only if the mixed-type operators $D^\pm_\mp$ are
\begin{equation}\label{eq:mixedcanonical}
D_{e_1^-}e_2^+ = [e_1^-,e_2^+]^+, \qquad D_{e_1^+}e_2^- = [e_1^+,e_2^-]^-.
\end{equation}
\end{enumerate}
\end{lemma}
\begin{proof}
Given $e_1,e_2,e_3 \in \Gamma(E)$, we have
\begin{equation}\label{eq:TDmixed}
T(e_1^-,e_2^+,e_3^+) = (D_{e_1^-}e_2^+ - [e_1^-,e_2^+],e_3^+),
\end{equation}
and therefore $ii)$ follows. To prove $i)$, we substract from \eqref{eq:TDmixed} the analogue expression for $T_{D'}(e_1^-,e_2^+,e_3^+)$.
\end{proof}

The second part of the previous lemma provides a weak analogue of the Koszul formula in pseudo-Riemannian geometry. The canonical mixed-type operators \eqref{eq:mixedcanonical} were first 
considered 
in \cite{H3} in the context of exact Courant algebroids, and used in \cite{G3} to construct a generalized geometric analogue of the Bismut connection in hermitian geometry. In our work, the main emphasis is on the space of torsion-free $V_+$-compatible generalized connections, which we denote
$$
\cD^0(V_+) = \cD(V_+) \cap \cD^0.
$$

\begin{proposition}\label{prop:existence}
Let $V_+$ be a generalized metric on a Courant algebroid $E$. Then, there exists a torsion-free generalized connection on $E$ which is compatible with $V_+$.
\end{proposition}
\begin{proof}
We construct first a $V_+$-compatible generalized connection $D$ on $E$ which has torsion of pure-type. For this, we define the mixed-type operators $D^\pm_\mp$ by \eqref{eq:mixedcanonical}. To define the \emph{pure-type operators}, we choose arbitrary metric connections $\nabla^+$ on $V_+$ and $\nabla^-$ on $V_-$, and set
$$
D_{e_1^\pm}e_2^\pm := \nabla^\pm_{\pi(e_1^\pm)}e_2^\pm.
$$
Finally, we define our torsion-free generalized connections $D^0$ by `killing the torsion', that is (see \eqref{eq:cyclic})
$$
D^0 = D - \frac{1}{3}T_D,
$$
where we use the metric $\la\cdot,\cdot\ra$ to regard
$$
T_{D} \in \Gamma(V_+^* \otimes \mathfrak{o}(V_+)) \oplus \Gamma(V_-^* \otimes \mathfrak{o}(V_-)).
$$
Crucially, the pure-type condition on $T_D$ implies that $D^0$ is $V_+$-compatible.
\end{proof}

Given a generalized metric, the torsion-free condition does not determine a compatible generalized connection uniquely. Unlike in pseudo-Riemannian geometry, the \emph{many Levi-Civita generalized connections} form an affine space, modelled on the pure-type mixed symmetric $3$-tensors
$$
\Sigma^+ \oplus \Sigma^-
$$
where
$$
\Sigma^\pm =  \Gamma(V_\pm^{\otimes 3}) \cap \Sigma,
$$
and $\Sigma$ is as in \eqref{eq:sigma}. There are canonical splittings
\begin{equation}\label{eq:splittingyoungpure}
\Sigma^\pm = \Sigma_0^\pm \oplus \Gamma(V_\pm),
\end{equation}
where the first summand corresponds to `trace-free' elements, in analogy with \eqref{eq:splittingyoung}. To see this--for $V_+$, say--let $r_+$ denote the rank of $V_+$, let $\{e_i^+\}$ be an orthogonal local frame for $V_+$ and let $\{\tilde e_i^+\} \subset \Gamma(V_+)$ be the `dual frame', defined so that $\la e_i^+, \tilde e_j^+ \ra = \delta_{ij}$. Then, on $V_+$ the splitting \eqref{eq:splittingyoungpure} corresponds to the following direct sum decomposition
\begin{equation}\label{eq:decompositionchipm}
\chi^+ = \chi^+_0 + \chi_+^{e^+}
\end{equation}
for a general element $\chi^+ \in \Gamma(V_+^* \otimes \mathfrak{o}(V_+))$, where $\chi_0^+$ is such that
$$
\sum_{i=1}^{r_+} (\chi_0^+)_{e_i^+}\tilde e_i^+ = 0, 
$$
and
$$
(\chi^{e^+}_+)_{e_1^+}e_2^+ = \la e_1^+,e_2^+ \ra e^+ - \la e^+,e_2^+ \ra e_1^+
$$
with 
$$
e^+ = \frac{1}{r_+ -1}\sum_{i=1}^{r_+} \chi_{e_i^+}\tilde e_i^+.
$$

\subsection{Canonical Dirac operators}\label{sec:ggdirac}

By Lemma \ref{lem:mixedfixed}, the freedom in the construction of a torsion-free generalized connection compatible with $V_+$ corresponds to the choice of pure-type operators
\begin{equation}\label{eq:LCpure}
D^+_+ \colon \Gamma(V_+) \to \Gamma(V_+^* \otimes V_+), \qquad D^-_- \colon \Gamma(V_-) \to \Gamma(V_-^* \otimes V_-).
\end{equation}
Using spinorial geometry, in this section we construct a pair of Dirac-type operators which are independent of choices, once we fix a divergence operator $div$ on the Courant algebroid $E$ (see Lemma \ref{lem:weylfixed}). 

Let $Cl(V_+)$ and $Cl(V_-)$ denote the bundles of Clifford algebras of $V_+$ and $V_-$, respectively. We assume that $Cl(V_\pm)$ admit irreducible Clifford modules $S_\pm$ such that $\Gamma(\End S_\pm) \cong \Gamma(Cl(V_\pm))$, that we fix in the sequel. Furthermore, we assume that the line bundles $(\det S_\pm^*)^{\frac{1}{r_{S_\pm}}}$ exist, and denote
\begin{equation}\label{eq:spinortwistedpm}
\mathcal{S}_\pm = S_\pm \otimes (\det S_\pm^*)^{\frac{1}{r_{S_\pm}}},
\end{equation}
where $r_{S_\pm}$ denotes the rank of $S_\pm$.

With these assumptions, similarly as in Section \ref{subsec:genoper}, a generalized connection $D \in \cD^0(V_+)$, induces canonically a pair of Dirac-type operators
\begin{align*}
\slashed D^+ & \colon \Gamma(\mathcal{S}_+) \to \Gamma(\mathcal{S}_+), \qquad \slashed D^- \colon \Gamma(\mathcal{S}_-) \to \Gamma(\mathcal{S}_-),
\end{align*}
given explicitly by (e.g. for $\alpha \in \Gamma(\mathcal{S}_+)$)
$$
\slashed D^+ \alpha = \frac{1}{2}\sum_{i=1}^{r_+} \tilde e_i^+ \cdot D_{e_i^+} \alpha.
$$

In the following lemma we prove that these natural operators are uniquely determined by the generalized metric, once we Weyl gauge fix the generalized connections by a choice of divergence operator on $E$ (see Definition \ref{def:divergenceop}). This provides a weak analogue of the Fundamental Lemma of pseudo-Riemannian geometry, that we shall use in Section \ref{subsec:Ricci} to study the Ricci tensor in generalized geometry. Let $div \colon \Gamma(E) \to C^\infty(M)$ be a divergence operator on $E$, and consider the space of $V_+$-compatible generalized connections with vanishing torsion and fixed divergence
\begin{equation}\label{eq:DVdiv}
\cD(V_+,div) = \{D \in \cD(V_+) \; | \; T_D = 0, \; div_D = div\}.
\end{equation}
The existence of an element in $\cD(V_+,div)$ follows by Proposition \ref{prop:existence}, along the lines of the proof of Lemma \ref{lem:weylfixed}. By this lemma, $\cD(V_+,div)$ is an affine space modelled on 
$$
\Sigma^+_0 \oplus \Sigma^-_0.
$$

\begin{lemma}\label{lem:dDpm}
The Dirac operators $\slashed D^+$ and $\slashed D^-$ are independent of the choice of $D \in \cD(V_+,div)$.
\end{lemma}

\begin{proof}
Let $D \in \cD(V_+,div)$ and consider a torsion-free $V_+$-compatible generalized connection $D' = D + \chi \in \cD^0(V_+)$ with
$$
\chi = \chi^+ + \chi^- \in \Sigma^+ \oplus \Sigma^-.
$$
Using the decomposition in \eqref{eq:decompositionchipm} there exists $e^\pm \in \Gamma(V_\pm)$ such that
$$
\chi^\pm = \chi^\pm_0 + \chi_\pm^{e^\pm}.
$$
Then, we obtain
$$
div_{D'} = div - (r_+ - 1)\la e^+,\cdot \ra - (r_- - 1) \la e^-, \cdot \ra
$$
and therefore $div_{D'} = div$ implies $e^+ = e^- = 0$. Define $\sigma \in \Gamma(\Lambda^3 E^*)$ by
$$
\sigma(e_1,e_2,e_3) = \la \chi_{e_1}e_2,e_3\ra,
$$
and note that the total skew-symmetrization $c.p. (\sigma)$ vanishes, since $D'$ and $D$ are torsion-free. Decomposing $\sigma = \sigma^+ + \sigma^-$ in pure-types, the result follows from
$$
\slashed D'^\pm = \slashed D^\pm - \frac{1}{4} c.p.(\sigma^\pm) = \slashed D^\pm.
$$

\end{proof}

\subsection{The case of exact Courant algebroids}\label{subsec:example}
In this section we discuss in detail our general construction in the special case of Example \ref{ex:divergence}. Let $E$ be an exact Courant algebroid over a manifold $M$ of dimension $n$ with \v Severa class $[H] \in H^3(M,\RR)$, given by a short exact sequence
\begin{equation}\label{eq:exactseqE}
0 \to T^* \to E \to T \to 0.
\end{equation}
Consider a generalized metric
$$
E = V_+ \oplus V_-
$$
of signature $(n,0)$. Then, $V_+$ induces an isotropic splitting of \eqref{eq:exactseqE} and an isomorphism of $E$ with the Courant algebroid in Example \ref{ex:divergence}, for a suitable choice of closed $3$-form $H \in [H]$. In this splitting, the generalized metric takes the simple form \cite{G3} 
$$
V_\pm = \{X \pm g(X): X \in T\},
$$
where $g$ is the Riemannian metric induced by the isomorphism $\pi_+ \colon V_+ \to T$.

Define connections on $T$ with skew torsion, compatible with the metric $g$,
given by
\begin{equation}\label{eq:nablapm}
    \nabla^\pm = \nabla^g \pm \frac{1}{2}g^{-1}H,
\end{equation}
where $\nabla^g$ denotes the Levi-Civita connection of the metric $g$
on $M$. As proved in \cite{G3}, one has the following formula for the canonical mixed-type operators \eqref{eq:mixedcanonical}: 
\begin{equation}\label{eq:bismutmix}
\begin{split}
[X^-,Y^+]^+ &= \frac{1}{2}\(\nabla^+_XY\)^+,\\
[X^+,Y^-]^- &= \frac{1}{2}\(\nabla^-_XY\)^-,
\end{split}
\end{equation}
where $X, Y \in \Gamma(T)$. Following the proof of Proposition \ref{prop:existence}, we use the isomorphisms $\pi_\pm \colon V_\pm \to T$ to regard $\nabla^\pm$ as metric connections on $V_\pm$, and define a $V_+$-compatible connection with pure-type torsion. By doing so, we obtain the \emph{Gualtieri-Bismut connection} with pure-type operators
\begin{equation}\label{eq:bismutpure}
\begin{split}
D^B_{X^+}Y^+ & = \frac{1}{2}\(\nabla^+_XY\)^+,\\
D^B_{X^-}Y^- & = \frac{1}{2}\(\nabla^-_XY\)^-,
\end{split}
\end{equation}
and torsion $T_{D^B} = \pi_+^*H + \pi_-^*H$. The torsion-free generalized connection $D^0 \in \cD^0(V_+)$ constructed from $D^B$ by
$$
D^0 = D^B - \frac{1}{3}T_{D^B}
$$
has pure-type operators
\begin{equation}\label{eq:LCpureex}
\begin{split}
D^0_{X^+}Y^+ & = \frac{1}{2}\(\nabla^{+1/3}_XY\)^+,\\
D^0_{X^-}Y^- & = \frac{1}{2}\(\nabla^{-1/3}_XY\)^-,
\end{split}
\end{equation}
where $\nabla^{\pm1/3}$ denote the metric connection with skew-symmetric torsion
\begin{equation}\label{eq:nabla+3}
    \nabla^{\pm1/3} = \nabla^g \pm \frac{1}{6}g^{-1}H.
\end{equation}

Consider now the generalized connection $D$ induced by $\nabla^g \oplus \nabla^{g*}$, as in Example \ref{ex:divergence}. It is easy to see that $D^B$ and $D^0$ differ from $D$ by totally skew-symmetric elements in $\Gamma(E^* \otimes \mathfrak{o}(E))$ (see \cite{GF}) and therefore
\begin{equation}\label{eq:divergenceex}
div_{D^B}(e) = div_{D^0}(e) = \mu_g^{-1}L_{\pi(e)}\mu_g,
\end{equation}
where $\mu_g \in \Gamma(|\det T^*|)$ is the Riemannian $1$-density. Denote $div = div_{D^0}$ and consider
\begin{equation}\label{eq:divergencevarphiex}
div^\varphi = div - \la \varphi, \cdot \ra,
\end{equation}
for a $1$-form $\varphi \in \Gamma(T^*)$. Then, we can construct an element $D^\varphi \in \cD(V_+,div^\varphi)$ 
by
\begin{equation}\label{eq:LCvarphi}
D^\varphi = D^0 + \frac{1}{n-1}\Big{(}\chi_+^{\varphi^+} + \chi_-^{\varphi^-}\Big{)},
\end{equation}
with pure-type operators
\begin{equation}\label{eq:LCvarphipure}
\begin{split}
D^\varphi_{X^+}Y^+ & = \frac{1}{2}\(\nabla^{+1/3}_XY + \frac{1}{2(n-1)}(g(X,Y)g^{-1}\varphi - \varphi(Y)X)\)^+,\\
D^\varphi_{X^-}Y^- & = \frac{1}{2}\(\nabla^{-1/3}_XY + \frac{1}{2(n-1)}(g(X,Y)g^{-1}\varphi - \varphi(Y)X)\)^-.
\end{split}
\end{equation}
Assuming that $M$ is a spin manifold, the canonical Dirac operators $\slashed D^{\varphi+}$ and $\slashed D^{\varphi-}$ associated to $(V_+,\slashed d)$ in Lemma \ref{lem:dDpm} exist, and are given by
\begin{equation}\label{eq:LCvarphiDirac}
\begin{split}
\slashed D^{\varphi\pm} & = \slashed \nabla^{\pm 1/3} + \frac{1}{4}\varphi \cdot.
\end{split}
\end{equation}

Consider now the spinor bundle $S = \Lambda^* T^*$ and the twisted spinor bundle
$$
\mathcal{S} = \Lambda^* T^* \otimes |\det T|^{1/2},
$$
and identify $\mathcal{S} \cong \mathbb{S}$ using $\mu^{\frac{1}{2}}_g$ (see Example \ref{example:dirac}). 
A calculation shows that
$$
\slashed D^B = d - H\wedge - \frac{1}{4}T_{D^B},
$$
and therefore the generalized connection $D^\varphi$ defined by \eqref{eq:LCvarphipure} satisfies 
$$
\slashed D^\varphi = d - H \wedge + \frac{1}{4}\varphi \wedge,
$$
which in particular implies $\slashed D^0 = \slashed d_0$. Note that, by Example \ref{example:dirac}, $\slashed D^\varphi$ is a generating operator if and only if $d \varphi = 0$.


\section{The Ricci tensor}\label{sec:Ricci}

\subsection{Curvature and the First Bianchi Identity}\label{sec:curvature}

Let $E$ be Courant algebroid over a smooth manifold $M$. We fix a generalized metric $E = V_+ \oplus V_-$ of arbitrary signature. All the generalized connections $D$ on $E$ considered in this section are compatible with $V_+$. Given such a generalized connection $D$ there are two curvature operators
\begin{align*}
R_D^\pm &\in \Gamma(V_\pm^* \otimes V_\mp^* \otimes \mathfrak{o}(V_\pm)),
\end{align*}
defined by
\begin{equation}\label{eq:curvature}
\begin{split}
R_D^+(e_1^+,e_2^-)e_3^+ = D_{e_1^+}D_{e_2^-}e_3^+ - D_{e_2^-}D_{e_1^+} e_3^+ - D_{[e_1^+,e_2^-]}e_3^+,\\
R_D^-(e_1^-,e_2^+)e_3^- = D_{e_1^-}D_{e_2^+}e_3^- - D_{e_2^+}D_{e_1^-} e_3^- - D_{[e_1^-,e_2^+]}e_3^-.
\end{split}
\end{equation}

The curvatures of a torsion-free generalized connection are not an invariant of the generalized metric $V_+$, but rather depend on $D$ via the choice of pure-type operators \eqref{eq:LCpure}. Our next objective is to define curvature quantities which only depend on the generalized metric $V_+$, after Weyl gauge fixing. Before we address this question in Section \ref{subsec:Ricci}, we prove an algebraic property of the curvatures of a torsion-free generalized connection, that we will use later in Lemma \ref{lem:Ricci}. 

\begin{proposition}[First Bianchi identity]\label{propo:bianchi}
Let $D$ be a torsion-free generalized connection compatible with $V_+$. Then, for any $e^\pm, e_1^\pm, e_2^\pm, e_3^\pm \in \Gamma(V_\pm)$we have
\begin{equation}\label{eq:bianchi}
c.p._{123} \la R_D^\pm(e_1^\pm,e^\mp)e_2^\pm,e_3^\pm\ra  = 0.
\end{equation}
\end{proposition}
\begin{proof}
We argue for $R^+_D$, as the other case is symmetric. We decompose
$$
c.p._{123} \la R_D^+(e_1^+,e^-)e_2^+,e_3^+\ra = I + J + K
$$
where
\begin{align*}
I & = c.p._{123} \la D_{e_1^+}([e^-,e_2^+]^+),e_3^+\ra,\\
J & = - c.p._{123} \la [e^-,D_{e_1^+} e_2^+],e_3^+\ra,\\
K & = - c.p._{123} \la D_{[e_1^+,e^-]}e_2^+,e_3^+\ra.
\end{align*}
Using that $T_D = 0$ and formula \eqref{eq:torsionG} we have 
\begin{align*}
K & = c.p._{123}  - \la D_{e_2^+}[e_1^+,e^-],e_3^+\ra - \la [[e_1^+,e^-],e_2^+],e_3^+\ra + \la D_{e_3^+}[e_1^+,e^-],e_2^+\ra\\
& = c.p._{123}  \la[e_1^+,e^-],D_{e_2^+} e_3^+ - D_{e_3^+} e_2^+ \ra - \la [[e_1^+,e^-],e_2^+],e_3^+\ra\\
&  - \pi(e_2^+)(\la[e_1^+,e^-],e_3^+\ra) + \pi(e_3^+)(\la [e_1^+,e^-],e_2^+\ra),\\
\end{align*}
where in the last equality we used the compatibility between $D$ and $\la\cdot,\cdot\ra$. Using again this last property, we also have
\begin{align*}
I & = c.p._{123}  \pi(e_1^+)(\la[e^-,e_2^+],e_3^+\ra) - \la[e^-,e_2^+],D_{e_1^+}e_3^+\ra,\\
J & = c.p._{123} - \pi(e^-)(\la[e_1^+,e_2^+],e_3^+\ra) + \la[e^-,e_3^+],D_{e_1^+}e_2^+\ra,
\end{align*}
where in the last equality we used property (C4) of the bracket (see Section \ref{subsec:connection})). From this, using property (C5) of the bracket--which 
implies $[e^-,e_i^+] = - [e_i^+,e^-]$--, we obtain
\begin{align*}
I + J + K & =  - \la [[e_1^+,e^-],e_2^+], e_3^+\ra + \la [[e_2^+,e^-],e_1^+], e_3^+\ra - \la [[e_3^+,e^-],e_1^+], e_2^+\ra\\
& + \pi(e_2^+)(\la [e^3_+,e^-],e_1^+\ra) - \pi(e_1^+)(\la [e^3_+,e^-],e_2^+\ra) - \pi(e_3^+)(\la [e_2^+,e^-],e_1^+\ra)\\
& - \pi(e^-)(T_D(e_1^+,e_2^+,e_3^+)) - \pi(e^-)(\la [e_1^+,e_2^+],e^+_3\ra).
\end{align*}
Finally, using again (C5) we have an equality 
$$
\pi(e_3^+)(\la [e^-,e_2^+],e_1^+\ra) = \la e_3^+, [e_1^+,[e^-,e_2^+]] + [[e^-,e_2^+],e_1^+]\ra,
$$
and from $T_D = 0$ we conclude
$$
I + J + K = \la [[e^-,e_1^+],e_2^+] + [e_1^+,[e^-,e_2^+]] - [e^-,[e_1^+,e_2^+]], e_3^+\ra,
$$
which vanishes by the Jacobi identity (property (C1)).
\end{proof}

The \emph{algebraic Bianchi identity} \eqref{eq:bianchi} for torsion-free generalized connections is a remarkable property, as their construction typically involves standard connections with skew-torsion in the tangent bundle of $M$. Recall that a torsionful connection on a manifold satisfies complicated Bianchi identities which involve first-order derivatives of its torsion (see e.g. \cite[Appendix]{Ferreira2}). 
We illustrate this with the following example.

\begin{example}\label{example:Curvature}
Let $V_+$ be a Riemannian generalized metric on an exact Courant algebroid $E$ over $M$, as considered in Section \ref{subsec:example}. Using the isomorphisms $\pi_\pm \colon V_\pm \to T$ we can identify the curvatures of the Gualtieri-Bismut generalized connection with
$$
R_{D^B}^\pm \cong R_{\nabla^\pm},
$$
where $R_{\nabla^\pm}$ is the (standard) curvature of the metric connection with skew-symmetric torsion $\pm H$ in \eqref{eq:nablapm} (see \cite[p. 4]{FrIv} for an explicit formula). Note that the first Bianchi identity for $R_{\nabla^\pm}$ involves covariant derivatives of the form $\nabla^g H$, as well as quadratic terms in $H$ (see e.g. \cite[Appendix]{Ferreira2}). 

Considering now the torsion-free generalized connection $D^\varphi \in \cD(V_+,\slashed d)$ defined by \eqref{eq:LCvarphi}, its curvature is given by \cite{GF} (e.g. for $R_{D^\varphi}^+$)
\begin{align*}
R_{D^\varphi}^+(X^+,Y^-)Z^+ & = \frac{1}{4}\Big{(}R^{1,3}(X,Y)Z + \frac{1}{4(n-1)}\(g(X,Z)\nabla^{+}_Y\varphi - i_Z\(\nabla^+_Y\varphi\)X\)\Big{)}^+
\end{align*}
where
\begin{align*}
R^{1,3}(X,Y)Z & = R^g(X,Y)Z + g^{-1}\Bigg{(}\frac{1}{2}(\nabla^g_XH)(Y,Z,\cdot) - \frac{1}{6}(\nabla^g_YH)(X,Z,\cdot)\\
& + \frac{1}{12}H(X,g^{-1}H(Y,Z,\cdot),\cdot) - \frac{1}{12}H(Y,g^{-1}H(X,Z,\cdot),\cdot)\\
& - \frac{1}{6}H(Z,g^{-1}H(X,Y,\cdot),\cdot)\Bigg{)}.
\end{align*}
The tensor $R^{1,3}$ is a hybrid of the \emph{second covariant derivatives} for $\nabla^+,\nabla^{1/3}$, and $\nabla^-$ with the remarkable property that satisfies the algebraic Bianchi identity \eqref{eq:bianchi}. More invariantly, it can be written as
\begin{equation}\label{eq:R13}
\begin{split}
R^{1/3}(X,Y)Z & = \nabla^{1/3}_X\nabla^+_YZ - \nabla^+_Y\nabla^{1/3}_XZ + \nabla^{1/3}_{\nabla^+_YX}Z - \nabla^+_{\nabla^-_XY}Z.
\end{split}
\end{equation}
\end{example}

\subsection{Definition of the Ricci tensor}\label{subsec:Ricci}

In order to define the Ricci tensor of a generalized metric, we introduce next the definition of Ricci tensor for a metric-compatible generalized connection, due to Gualtieri (see e.g. \cite{GF,Streets}).

\begin{definition}\label{def:RicciD}
The \emph{Ricci tensors} $Ric_D^\pm \in \Gamma(V_\mp^* \otimes V_\pm^*)$ of a generalized connection compatible with $V_+$ are defined by
\begin{equation}\label{eq:Ricci}
Ric_D^\pm(e^\mp,e^\pm) = \tr (d^\pm \to R_D^\pm(d^\pm,e^\mp)e^\pm),
\end{equation}
for $e^\pm \in \Gamma(V_\pm)$.
\end{definition}

Our next result investigates the variation of the Ricci tensors when we deform a torsion-free generalized connection in $\cD(V_+)$, while preserving the pure-type condition of the torsion and the divergence (see Definition \ref{def:divergenceop}). 

\begin{proposition}\label{propo:Riccitorsion}
Assume that $M$ is positive-dimensional. Let $(V_+,div)$ be a pair given by a divergence operator and a generalized metric on a Courant algebroid $E$. Let $D, D' \in \cD(V_+)$ be generalized connections with divergence $div_D = div = div_{D'}$. If $T_D = 0$ and $T_{D'}$ is of pure-type then $Ric^\pm_D = Ric^\pm_{D'}$. In particular, $Ric^\pm_D = Ric^\pm_{D'}$ for any pair $D,D' \in \cD(V_+,div)$ (see \eqref{eq:DVdiv}).
\end{proposition}
\begin{proof}
Setting $\chi = D' - D$, the condition $div_D = div_{D'}$ combined with $T_{D'}$ being of pure-type imply that
$$
\chi = \chi^+ + \chi^-,
$$
where (see \eqref{eq:splittingyoungpure})
$$
\chi^\pm = \chi^\pm_0 + \sigma^\pm \in \Sigma_0^\pm \oplus \Lambda^3 V_\pm.
$$
Given $e^\pm \in \Gamma(V_\pm)$ we calculate
\begin{equation}\label{eq:RicDD}
\begin{split}
Ric_{D'}^+(e^-,e^+) & = \sum_{i=1}^{r_+} \la \tilde e_i^+, R_{D'}^+(e_i^+,e^-)e^+\ra\\
& = Ric_{D}^+(e^-,e^+) + \sum_{i=1}^{r_+} \la \tilde e_i^+, \chi^+_{e_i^+}(D_{e^-}e^+) - D_{e^-} (\chi^+_{e_i^+}e^+) - \chi^+_{[e_i^+,e^-]^+}e^+\ra\\
& = Ric_{D}^+(e^-,e^+) + \sum_{i=1}^{r_+} \Big{(} - \la \chi^+_{e_i^+}\tilde e_i^+,D_{e^-}e^+\ra + \la \chi^+_{e_i^+}  e^+,D_{e^-}\tilde e_i^+\ra\\
& + \iota_{\pi(e_-)} d\la \chi^+_{e_i^+}\tilde e_i^+,e^+\ra - \la \chi^+_{[e^-,e_i^+]^+} \tilde e_i^+,e^+\ra \Big{)}\\
& = Ric_{D}^+(e^-,e^+) + \sum_{i=1}^{r_+} \Big{(} \la \chi^+_{e_i^+} e^+,[e^-,\tilde e_i^+]\ra  - \la \chi^+_{[e^-,e_i^+]^+} \tilde e_i^+,e^+\ra \Big{)},
\\
& = Ric_{D}^+(e^-,e^+) + \sum_{i=1}^{r_+} \Big{(} \la (\chi^+_0)_{e_i^+} e^+,[e^-,\tilde e_i^+]^+\ra  - \la (\chi^+_0)_{[e^-,e_i^+]^+} \tilde e_i^+,e^+\ra \Big{)}
\end{split}
\end{equation}
where have used $\sum_{i=1}^{r_+} \chi^+_{e_i^+}\tilde e_i^+ = 0$, Lemma \ref{lem:mixedfixed}, and that $\sigma^\pm \in \Lambda^3 V_\pm$. To conclude, given $x \in M$ and $e^- \in V_{-|x}$, using that $M$ is positive-dimensional we can take a sequence $\{e^-_k\}$ of elements $e^-_k \in V_{-|x}$ which approximate $e^-$ and such that $\pi(e^-_k) \neq 0$. We will prove that 
$$
Ric_{D'}^+(e^-_k,e^+) = Ric_{D}^+(e^-_k,e^+)
$$
holds for all $k$, and hence the statement will follow since $Ric_{D'}^+$ and $Ric_{D}^+$ are bilinear. For this, we construct a sequence of special orthogonal frames in order to evaluate formula \eqref{eq:RicDD} with $e^-$ replaced by $e^-_k$. Choose an orthogonal frame $\{e_{i,k}^+\}$ around $x$ for $V_+$ which satisfies $D_{e^-_k}e_{i,k}^+ = 0$ at $x$. This frame can constructed by smooth extension of the parallel transport for the $V_-$-connection $D^+_-$ along a curve starting at $x$ with initial velocity $\pi(e^-_k)$ (see e.g. \cite{Loja}). The proof now follows from the previous formula, using that $[e^-_k,e_{i,k}^+]^+ = D_{e^-_k}e_{i,k}^+$ and $[e^-_k,\tilde e_{i,k}^+]^+ = D_{e^-_k}\tilde e_{i,k}^+$, combined with the fact that $\tilde e_{i,k}^+ = \epsilon_i e_{i,k}^+$ for some constant $\epsilon_i \in \RR$. The last part of the statement follows from Lemma \ref{lem:weylfixed}.
\end{proof}

This lead us to our definition of the Ricci tensor.

\begin{definition}\label{def:Ricci}
Let $div$ be a divergence operator on the Courant algebroid $E$, and let $V_+ \subset E$ be a generalized metric on $E$. The \emph{Ricci tensors} 
$$
Ric^\pm \in \Gamma(V_\mp^* \otimes V_\pm^*)
$$ 
of the pair $(V_+,div)$ are defined by $Ric^\pm = Ric^\pm_D$ for any choice of generalized connection $D \in \cD(V_+,div)$.
\end{definition}

\begin{remark}\label{rem:quadLie}
In the case that $M$ is a point, that is, when $E$ is a quadratic Lie algebra, the proof of Proposition \ref{propo:Riccitorsion} using parallel transport does not apply and the Ricci tensors $Ric^\pm_D$ depend upon the choice of $\chi^\pm_0 \in \Sigma_0^\pm$. A conceptual explanation of this fact is as follows: given a vector space $W$, thought of as a bundle over a point, a $W$-connection on $W$ is an element $D \in W^* \otimes \End W$. The unique trivial path on the base can be lifted to a constant path $w_0 \in W$. Given $w \in W$ the parallel transport of $w$ along $w_0$ may not exist, unless $D_{w_0} w = 0$. This suggests that formula \eqref{eq:Ricci} is not the best approach to Definition \ref{def:Ricci}. A more fundamental approach to the Ricci tensor using spinors is provided by Lemma \ref{lem:Ricci} below (see Example \ref{example:Riccioverpoint}).
\end{remark}


We continue our discussion with an alternative point of view on the Ricci tensor, using the canonical Dirac operators constructed in Section \ref{sec:ggdirac}. For this, we assume that $V_\pm$ admit spinor bundles $S_\pm$ and consider the twisted spinor bundles $\mathcal{S}_\pm$ endowed with the canonical Dirac operators $\slashed D^\pm$ on $\Gamma(\mathcal{S}_\pm)$ as in Lemma \ref{lem:dDpm}.

\begin{lemma}\label{lem:Ricci}
Assume that $M$ is positive-dimensional. Let $V_+ \subset E$ be a generalized metric and let $div$ be a divergence operator on $E$. Then, 
the Ricci tensors in Definition \ref{def:Ricci} can be calculated by
\begin{equation}\label{eq:Ricci+-op}
\iota_{e^\mp} Ric^\pm \cdot \alpha = 4\Big{(}\slashed D^\pm D_{e^\mp} - D_{e^\mp}\slashed D^\pm - 2 \sum_{i=1}^{r_\pm} \tilde e_i^\pm \cdot D_{[e_i^\pm,e^\mp]^\mp}\Big{)}\alpha,
\end{equation}
where $\alpha \in \Gamma(\mathcal{S}_\pm)$ and $e^\mp \in \Gamma(V_\mp)$.
\end{lemma}

\begin{proof}
It is easy to see that the right hand side of \eqref{eq:Ricci+-op} is tensorial in $e^\mp \in \Gamma(V_\mp)$. To evaluate \eqref{eq:Ricci+-op}, we choose an orthogonal frame $\{e_i^+\}$ for $V_+$ around $x \in M$ which satisfies $D_{e^-}e_i^+ = 0$ at the point $x$. This frame can constructed as in the proof of Proposition \ref{propo:Riccitorsion}, by smooth extension of the parallel transport for the $V_-$-connection $D^+_-$ along a curve starting at $x$ with initial velocity $\pi(e^-)$ (see e.g. \cite{Loja}). With the conventions above we have
\begin{align*}
([\slashed D^\pm,D_{e^\mp}]\alpha)_{|x} &= \frac{1}{2}\sum_{i=1}^{r_+} \tilde e_i^+ \cdot (D_{e_i^+}D_{e^-} - D_{e^-} D_{e_i^+} - D_{[e_i^+,e^-]}) \alpha\\
& = \frac{1}{2}\sum_{i=1}^{r_+} \tilde e_i^+ \cdot R^+_D(e_i^+,e^-) \cdot \alpha\\
& = \frac{1}{4}\tr (d^+ \to R^+_D(d^+,e^-))\cdot \alpha \\
& - \frac{1}{4} \sum_{i < j < k} \(c.p._{ijk} \la R^+_D(e_i^+,e^-)e_j^+,e_k^+\ra\) e_i^+ \cdot e_j^+ \cdot e_k^+ \cdot \alpha,
\end{align*}
where in the first equality we have used that $[e^-,e_i^+]^+ = D_{e^-}e_i^+$ by Lemma \ref{lem:mixedfixed}, and $(D_{e^-}e_i^+)_{|x} = 0$ by the choice of frame. The proof follows from the algebraic Bianchi identity in Proposition \ref{propo:bianchi}.
\end{proof}

\begin{remark}\label{rem:RicciWaldram}
Our formula \eqref{eq:Ricci+-op} is an interpretation of a formula for the Ricci tensor of a generalized metric in the physics literature \cite{CSW}, claimed without proof. Our proof relies in the Bianchi identity \eqref{propo:bianchi} for the generalized curvature of a torsion free-generalized connection, which seems to be new.
\end{remark}

Since the right hand side of \eqref{eq:Ricci+-op} is tensorial in $e^\mp$ 
and the operators $\slashed D^+$ and $D_-^+$ 
only depend on $(V_+,div)$ (see Lemma \ref{lem:mixedfixed} and Lemma \ref{lem:dDpm}), formula \eqref{eq:Ricci+-op} can be taken as an alternative definition of the Ricci tensors, without relying on Proposition \ref{propo:Riccitorsion}. For this, we can regard \eqref{eq:Ricci+-op} as a local formula, and therefore there is no obstruction for the existence of the spinor bundles. It is interesting to notice that, unlike Definition \ref{def:Ricci}, formula \eqref{eq:Ricci+-op} yields a definition of the Ricci tensor in the case of quadratic Lie algebras which is independent of the choice of generalized connection $D \in \cD(V_+,div)$ (cf. Remark \ref{rem:quadLie}).
We finish this section with an explicit calculation of the Ricci tensors in two examples. Further explicit calculations of Ricci tensors on transitive Courant algebroids are provided in Section \ref{subsec:transitive}.

\begin{example}\label{example:Ricci}
From the formulae in Example \ref{example:Curvature}, we obtain the following expression for the Ricci tensors of $(V_+,div^\varphi)$ (see \eqref{eq:divergencevarphiex}) 
\begin{equation}\label{eq:Riccipmexact}
Ric^\pm = Ric_g - \frac{1}{4} H \circ H \mp \frac{1}{2}d^*H - \frac{1}{4}\nabla^\pm\varphi,
\end{equation}
where $H \circ H$ denotes the symmetric tensor $(H \circ H)_{ij} = g^{rs}g^{kl}H_{irk}H_{jsl}$. We notice an important property of these tensors--strongly reminiscent of the symmetric property of the Ricci tensor in Riemannian geometry--for which we have no Lie theoretic proof: the decomposition $Ric^\pm = h^\pm + b^\pm$ into symmetric and skew symmetric parts, respectively, satisfies:
\begin{equation}\label{eq:Riccisymmetry}
h^+ = h^-, \qquad b^+ = -b^-,
\end{equation}
provided that $d \varphi = 0$. Recall from Example \ref{example:dirac} that this last condition is equivalent to $\slashed D^\varphi$ being a Dirac generating operator on the exact Courant algebroid. Some consequences of this important symmetry will be discussed in Section \ref{sec:flow}. 
When $\varphi = 0$, $Ric^\pm$ equals the Ricci tensor of the connection $\nabla^\pm$, which corresponds to the Ricci tensor for the Gualtieri-Bismut connection $Ric^\pm_{D^B}$. We note here that $div_{D^B} = div_{D^0}$, in agreement with Proposition \ref{propo:Riccitorsion}.
\end{example}

\begin{example}\label{example:Riccioverpoint}
Assume that $M$ is a point, that is, $E$ is a quadratic Lie algebra. In this case, a divergence operator is just an element $e \in E$, and formula \eqref{eq:Ricci+-op} yields
\begin{equation}\label{eq:Ricciexp}
Ric^+(e_1^-,e_2^+) = \tr (d^+ \to [[e_1^-,d^+]^-,e_2^+]^+) - \la [e^+,e_1^-],e_2^+\ra.
\end{equation}
Incidentally, when $e = 0$ this coincides with formula \eqref{eq:Ricci} for the following choice of generalized connection with zero torsion and divergence
$$
D_{e_1}e_2 = [e_1^-,e_2^+]^+ + [e_1^+,e_2^-]^- + \frac{1}{3}[e_1^+,e_2^+]^+ + \frac{1}{3}[e_1^-,e_2^-]^-.
$$

\end{example}

\section{Ricci flow and the Killing spinor equations}\label{sec:flow}

\subsection{The Ricci flow in generalized geometry}


Let $E$ be a Courant algebroid over a smooth manifold $M$. Throughout this section, we change our perspective and consider that a generalized metric is given by an orthogonal endomorphism
$$
G \colon E \to E
$$
inducing a non-degenerate pairing $\la G \cdot, \cdot \ra$ and squaring to the identity $G^2 = \Id$. With this alternative definition (see Section \ref{sec:ggmetrics}), the orthogonal decomposition $E = V_ + \oplus V_-$ corresponds to the $\pm 1$-eigenbundle decomposition of $E$ with respect to $G$.

Given a smooth one-parameter family of generalized metrics $G_t$, the variation $\partial_t G$ defines a skew-symmetric endomorphism of $E$ by
$$
G_t \partial_t G \in \Gamma(\mathfrak{o}(E)),
$$ 
which exchanges the subbundles $V_+$ and $V_-$, yielding a pair of endomorphisms
$$
\partial_t G^+ \colon V_- \to V_+, 
$$
$$
\partial_t G^- \colon V_+ \to V_-.
$$

We are ready to introduce our notion of \emph{Ricci flow}. 

\begin{definition}\label{def:flow}
Let $E$ be a Courant algebroid over a smooth manifold $M$. We say that a one-parameter family $(G_t,div_t)$ of generalized metrics $G_t$ and divergence operators $div_t$ is a solution of the \emph{Ricci flow} if
\begin{equation}\label{eq:Ricciflowpm}
\partial_t G^+ = - 2 Ric^+_t,
\end{equation}
where $Ric^+_t$ is the Ricci tensor of $(G_t,div_t)$ as in Definition \ref{def:Ricci}.
\end{definition}

In the previous definition, the Ricci tensors $Ric^\pm$ of $(G,div)$ are regarded as endomorphisms
$$
Ric^\pm \colon V_\mp \to V_\pm,
$$
using the induced metric on $V_\pm$ to identify $V_\pm \cong V_\pm^*$. Note that the choice of $Ric^+$ instead of $Ric^-$ is just a matter of convention, since we are considering metrics with arbitrary signature. Building on the seminal work by Hamilton \cite{Hamilton}, short-time existence, higher derivative estimates, and a compactness theorem for the flow \ref{eq:Ricciflowpm} on a mild class of transitive Courant algebroids have been provided in \cite{He}.


We show next that the Ricci flow introduced in Definition \ref{def:flow} extends a previous notion of the Ricci flow on exact Courant algebroids by Gualtieri (see \cite{Streets}). To see this, we assemble the Ricci tensors $Ric^\pm$ into a unique endomorphism
\begin{equation}\label{eq:totalRicci}
Ric := Ric^+ - Ric^- \in \Gamma(\End E),
\end{equation}
that we will call the \emph{total Ricci tensor},
and consider the evolution equation
\begin{equation}\label{eq:ggflow}
\partial_t G = [Ric,G_t].
\end{equation}
When $E$ is an exact Courant algebroid and $\partial_t div = 0$, \eqref{eq:ggflow} coincides with Gualtieri's definition for the \emph{generalized Ricci flow}. Note that the right hand side of this equation is $-2(Ric^+_t + Ric^-_t)$, and therefore \eqref{eq:ggflow} is equivalent to the system
\begin{equation}\label{eq:ggflowpair}
\begin{split}
\partial_t G^+ = - 2Ric_t^+,\\
\partial_t G^- = - 2Ric_t^-.
\end{split}
\end{equation}
A priori, these equations are stronger than the Ricci flow \eqref{eq:Ricciflowpm}. The equivalence between the  flows \eqref{eq:Ricciflowpm} and \eqref{eq:ggflow} relies on an expected symmetry of the Ricci tensor in Definition \ref{def:Ricci}, for which we still do not have a general proof. To motivate the next definition, note that if $(G_t,div_t)$ is a solution of the system \eqref{eq:ggflowpair}, then the total Ricci tensor \eqref{eq:totalRicci} must be an element in $\Gamma(\mathfrak{o}(E))$, as so is $G_t \partial_t G$.

\begin{definition}\label{def:Riccisym}
A pair $(V_+,div)$ has the \emph{skew-symmetry property} if the total Ricci tensor \eqref{eq:totalRicci} is skew-orthogonal with respect to the ambient metric on $E$, that is,
$$
Ric \in \Gamma(\mathfrak{o}(E)).
$$
\end{definition}

The relevance of the skew-symmetry property is due to the following fact, which is a a straightforward consequence of the previous definition.

\begin{proposition}\label{prop:Riccisym}
A solution $(G_t,div_t)$ of the Ricci flow \eqref{eq:Ricciflowpm} is also a solution of the generalized Ricci flow \eqref{eq:ggflow} if and only if $(G_t,div_t)$ has the skew-symmetry property for all $t$.
\end{proposition}

On an exact Courant algebroid, a pair $(V_+,div)$--given by a generalized metric $V_+ \subset E$ such that $V_+ \cap T^* = \{0\}$ and a divergence operator $div$--has the skew symmetry property if and only if $div$ is induced by a Dirac generating operator. 
Note that $V_+ \cap T^* = \{0\}$ is the condition for $V_+$ to induce a metric $g$ on $M$. Our claim follows from an explicit general formula for the Ricci tensors in Example \ref{example:Ricciflow} below (see also Lemma \ref{lem:Riccitr}). 
More generally, the `if part' of this result is true for transitive Courant algebroids obtained by reduction (see Proposition \ref{prop:Riccisymtr}). On general grounds, when $E$ admits a spinor bundle we expect that a pair $(V_+,div)$ has the skew-symmetry property provided that $div$ is induced by a Dirac generating operator, but we have not been able to give a Lie-theoretic proof of this fact.

In the next example we give an explicit formula for the flows \eqref{eq:Ricciflowpm} and \eqref{eq:ggflow} in the case of exact Courant algebroids. A formula in the more general case of transitive Courant algebroids obtained by reduction is provided by \eqref{eq:Ricciflowtr}.

\begin{example}\label{example:Ricciflow}
Let $(G_t,div_t)$ be a solution of the generalized Ricci flow \eqref{eq:ggflow} on an exact Courant algebroid $E$, with $G_t$ inducing a metric $g_t$ on $M$ for all $t$. Fixing an isotropic splitting for $E$, the family $(G_t,div_t)$ can be written as
$$
G_t = e^{b_t} \left( \begin{array}{cc}
0 & g_t^{-1} \\
g_t & 0 \end{array}\right) e^{-b_t}
$$
$$
div_t(e) = \mu^{-1}_{g_t}L_{\pi(e)}\mu_{g_t} - \la e_t,e\ra
$$
where $b_t$ is the $2$-form which codifies the isotropic splitting determined by $G_t$, $H$ is the closed $3$-form in the fixed isotropic splitting and $e_t \in \Gamma(T \oplus T^*)$. 
Without loss of generality, we assume $b_{t_0} = 0$, and then the derivative of $G_t$ at times $t_0$ is
$$
\partial_t G = [\partial_t b,G_{t_0}] + \left( \begin{array}{cc}
0 & - g_t^{-1} \partial_t g_t g_t^{-1} \\
\partial_t g_t & 0 \end{array}\right)
$$
which implies that the two equations in the system \eqref{eq:ggflowpair} are equivalent to, respectively, (see \eqref{eq:Riccipmexact})
\begin{equation}\label{eq:flowex}
\begin{split}
\partial_t g_t - \partial_t b_t & = -2 Ric^+ = -2(Ric_{g_t} - \frac{1}{4} H_t \circ H_t - \frac{1}{2}d^*H_t - \frac{1}{4} \nabla^+\varphi_t),\\
\partial_t g_t + \partial_t b_t & = -2 Ric^- = -2(Ric_{g_t} - \frac{1}{4} H_t \circ H_t + \frac{1}{2}d^*H_t - \frac{1}{4} \nabla^-\sigma_t),
\end{split}
\end{equation}
where $e_t = \varphi_t^+ + \sigma_t^-$, for a pair of (families of) $1$-forms $\varphi_t, \sigma_t \in  \Gamma(T^*)$. The first equation is the Ricci flow of $(G_t,div_t)$ in Definition \ref{def:flow}. When $\varphi_t = \sigma_t$, that is, when $\pi(e_t) = 0$, and furthermore $d \varphi_t = 0$, the pair $(G_t,div_t)$ has the skew-symmetry property, and therefore the system \eqref{eq:flowex} reduces to one equation by \eqref{eq:Riccisymmetry}.
\end{example}

\begin{remark}\label{rem:dilaton}
In the case that $\varphi_t = 2d\phi_t$ is exact, the first equation in \eqref{eq:flowex} corresponds to the renormalization group flow in Type II string theory \cite{CFMP}, 
and therefore have an important physical motivation. In this context, 
the symmetric and skew-symmetric parts of the Ricci tensor $Ric^+$ are (essentially) the $\beta$-\emph{functions} for 
the gravitational field and the $B$-field, respectively, which control the conformal invariance of the effective
action in the sigma model approach to the theory. The function potential $\phi_t$ is the so-called \emph{dilaton field}, which plays an important role in T-duality (see Section \ref{subsec:dilatonshift}). 
\end{remark}

\subsection{Killing spinors in generalized geometry}\label{subsec:Killing}

The stationary points of the Ricci flow \eqref{eq:Ricciflowpm} correspond to \emph{Ricci flat pairs} $(V_+,div)$, that is, satisfying
\begin{equation}\label{eq:Ricciflat}
Ric^+ = 0.
\end{equation}
Being the Ricci flat condition a second order equation, it is natural to ask whether it can be regarded as an integrability condition for a first-order equation. This motivates our next definition. 


\begin{definition}\label{def:killing}
We say that a pair $(V_+,div)$ is a solution of the \emph{Killing spinor equations}, if there exists a spinor bundle $S_+$ for $V_+$ and a non-vanishing section of the twisted spinor bundle $\eta \in \Gamma(\mathcal{S}_+)$ such that
\begin{equation}\label{eq:killing}
\begin{split}
D^+_- \eta &= 0,\\
\slashed D^+ \eta &= 0,
\end{split}
\end{equation}
where $D^+_-$ and $\slashed D^+$ are the operators defined in Lemma \ref{lem:mixedfixed} and Lemma \ref{lem:dDpm}.
\end{definition}

Lemma \ref{lem:mixedfixed} and Lemma \ref{lem:dDpm} imply that these equations are independent of the choice of torsion-free generalized connection, as the operators $D^+_-$ and $\slashed D^+$ are canonically defined by the pair $(V_+,div)$. Definition \ref{def:killing} generalizes \cite[Definition 5.3]{grt} to arbitrary Courant algebroids, and it is inspired by the approach to the \emph{Killing spinor equations in supergravity} in \cite{CSW,GrMPrTo} (see Remark \ref{rem:killingname} and Remark \ref{rem:killingphys}).

In the next proposition we state the precise relation between solutions of the Killing spinor equations and Ricci flat pairs. This result, that follows as a direct consequence of Lemma \ref{lem:Ricci}, can be seen as an analogue of a classical result in Riemannian geometry (resp. pseudo-Riemannian geometry): metrics with parallel spinor fields are Ricci flat (resp. \emph{Ricci null}).

\begin{proposition}\label{prop:Ricciflat}
If $(V_+,div)$ is a solution of the Killing spinor equations \eqref{eq:killing}, then the image of the endomorphism $Ric^+ \colon V_- \to V_+$ consists of null vectors. In particular, if $V_+$ is positive definite then $Ric^+ = 0$.
\end{proposition}
\begin{proof}
By Lemma \ref{lem:Ricci} we have $\iota_{e_-}Ric^+ \cdot \eta = 0$. Since $\eta$ is non-vanishing, Clifford multiplication with $\iota_{e_-}Ric^+$ yields $|\iota_{e_-}Ric^+|^2 = 0$, and the result follows.
\end{proof}

Solutions of the Killing spinor equations \eqref{eq:killing} are very interesting objects, both from a geometric and from a physical point of view. To see this, let us focus on the case of a \emph{Riemannian generalized metric} $V_+$ on an exact Courant algebroid $E$ over $M$, that is, a generalized metric inducing a Riemannian metric on $M$ via $\pi_+ \colon V_ + \to  T$. In this context, the condition that $V_+$ admits an irreducible Clifford module in Definition \ref{def:killing} implies that $M$ admits a spin structure. Using the anchor map and the explicit formulae \eqref{eq:bismutmix} and \eqref{eq:LCvarphiDirac}, the Killing spinor equations \eqref{eq:killing} read
\begin{equation}\label{eq:killingexp}
\begin{split}
\nabla^+ \eta  = 0,\\
\slashed \nabla^{\frac{1}{3}} \eta + \frac{1}{4} \varphi \cdot \eta = 0.
\end{split}
\end{equation}
The equations \eqref{eq:killingexp} are equivalent, in low dimensions, to the \emph{Killing spinor equations in supergravity} \cite{FrIv} (see Remark \ref{rem:killingname} and Remark \ref{rem:killingphys}).
In the case $\varphi = 0$, Proposition \ref{prop:Ricciflat} combined with \eqref{eq:Riccipmexact} imply that any solution of \eqref{eq:killingexp} is an \emph{Einstein metric with skew torsion}, as studied in \cite{Ferreira2}. 

Using the explicit formula \eqref{eq:killingexp} combined with Proposition \ref{prop:Ricciflat}, we show next that solutions of the Killing spinor equations 
correspond to (standard) \emph{special holonomy metrics} on a compact manifold, provided that $\varphi$ is exact. Several proofs of this fact by different authors can be found in the mathematical physics literature (see e.g. \cite{IvanovPapadopoulos}), but we have included a (dimension independent) proof using \eqref{eq:killing} to illustrate our framework. The novelty here is that, independently of the dimension, metrics with 
parallel spinors are natural objects in generalized geometry; something which does not seem to have been pointed out before (see \cite{GF} for the case of $SU(3)$-holonomy metrics).

\begin{proposition}\label{prop:specialholonomy}
Let $E$ be an exact Courant algebroid over a smooth compact spin manifold $M$. Then, a solution $(V_+,div^\varphi)$ of the Killing spinor equations \eqref{eq:killing} with $\varphi$ exact is equivalent to a Riemannian metric with a parallel spinor.
\end{proposition}
\begin{proof}
The proof follows by explicit calculations. Denote by $g$ and $H$ the Riemannian metric and the closed $3$-form in the splitting $E \cong T \oplus T^*$ induced by $V_+$. By assumption $\varphi = 2d \phi$ in \eqref{eq:killing}, and by Proposition \ref{prop:Ricciflat} $(V_+,div^\varphi)$ is Ricci flat \eqref{eq:Ricciflat}. 
%
Contracting with $g$ in the formula for $Ric^+$ \eqref{eq:Riccipmexact} we obtain
$$
S_g - \frac{3}{2} |H|^2 + \frac{1}{2}\Delta_g\phi = 0,
$$
where we are using the conventions $|H|^2 = g^{ij}g^{kl}g^{st}H_{iks}H_{jlt}/6$ and $\Delta_g = d^*d$. Using now the Killing spinor equation \eqref{eq:killingexp} and Bismut's Lichnerowicz-type formula for the cubic Dirac operator \cite{Bismut} (
note the different sign convention for the Clifford algebra) 
$$
0 = 4(\slashed \nabla^{\frac{1}{3}} + \varphi/4 \cdot)^2 + 4(\nabla^+)^*\nabla^+) = - S_g - 2\Delta_g\phi  + |d\phi|^2 + \frac{1}{2}|H|^2.
$$
Adding the two scalar equations we obtain 
$$
|H|^2 + \frac{3}{2}\Delta_g \phi - |d\phi|^2 = 0.
$$
However, for any $\alpha\in \mathbb{R}$, and for $\Delta=d^* d$,
\begin{eqnarray*}
\Delta_g (e^{\alpha \phi})=-\alpha^2e^{\alpha\phi}|d\phi|^2+\alpha e^{\alpha\phi}\Delta_g \phi,
\end{eqnarray*}
and so
\begin{eqnarray*}
3\Delta(e^{\alpha\phi}) + \alpha(3\alpha-2)e^{\alpha\phi}|d\phi|^2 + 2\alpha e^{\alpha\phi}|H|^2 =0. 
\end{eqnarray*}
We can take $\alpha=1$ and integrate over $M$ to conclude that $H=0$ and $d\phi=0$. Finally, using the first equation in \eqref{eq:killingexp} it follows that $\nabla^g \eta = 0$. The converse is trivial, by taking the standard Courant algebroid $E = T \oplus T^*$ with $H = 0$ and $div$ the Riemannian divergence (see \eqref{eq:divergenceex}).
\end{proof}

An alternative treatment of special holonomy metrics using generalized geometry can be found in the work of Witt \cite{Witt}, via the study of $G \times G$-structures on exact Courant algebroids. In this alternative approach, a special holonomy metric is regarded (point-wise) as a diagonal structure $G \subset G \times G \subset O(n,n)$. We note however that an arbitrary Courant algebroid isomorphism may destroy the diagonal property, producing a general $G \times G$-structure. To the knowledge of the author, the Killing spinor equations \eqref{eq:killingexp} provide the first known `embedding' of special holonomy metrics onto generalized geometry. We will see some potential applications of this result in the next section.
In the case that $\varphi$ is closed but not exact (see Example \ref{example:dirac}) there are interesting solutions of the Killing spinor equations \eqref{eq:killing} in exact Courant algebroids, with $H \neq 0$ and $\varphi \neq 0$ \cite{grts}. 

More generally, for transitive Courant algebroids obtained by reduction, it has been proved in \cite{grt} that the system \eqref{eq:killing} corresponds to the Hull-Strominger system of partial differential equations (in the sense of \cite{GF2018,grts}). We postpone any further explanation about this interesting subject to Section \ref{subsec:transitive}, where we will study the behaviour of these equations under T-duality.

\begin{remark}\label{rem:killingphys}
In low dimensions, the equations \eqref{eq:killingexp} are equivalent to the \emph{killing-spinor equations} in compactifications of the common sector of ten-dimensional supergravity \cite{FrIv}, which may serve as motivation for our name. This particular way of writting the equations--which seems to be crucial for the relation with generalized geometry--was first considered in \cite{GrMPrTo} in the context of type II supergravity (in order to obtain an expression which is half-independent of the Ramond-Ramond fields). 
\end{remark}

\begin{remark}\label{rem:killingname}
When $E$ is exact, a killing spinor in the sense of Definition \ref{def:killing} is a hybrid between the notions of parallel spinor 
and harmonic spinor for connections with skew-symmetric torsion. 
The name \emph{Killing spinor} for a solution of the equations \eqref{eq:killing} may be misleading, as there is a well-established notion of Killing spinor in pseudo-Riemannian geometry. 
The alternative given by \emph{generalized Killing spinor} does not do any better, according to the classical notion by Friedrich, Kim, B\"ar, Gauduchon, and Moroianu. Despite this, in the most interesting examples (see Section \ref{subsec:transitive}) the equations \eqref{eq:killing} are motivated by physics and have a well-established name. 
\end{remark}

\section{T-duality in generalized geometry}\label{sec:dualities}

\subsection{T-duality and main result}\label{subsec:result}

We go now for the mainstream of our development. In this section we investigate a method to produce new solutions of the partial differential equations introduced in Section \ref{sec:flow} based on T-duality. The idea has its origins in the string theory literature in the work of Buscher \cite{Buscher1}, and was developed further by Ro\v cek and Verlinde \cite{RoVer}. 
Our main theorem builds on an important result by Cavalcanti and Gualtieri \cite{CaGu}, which states that \emph{topological T-duality} for principal torus bundles--as defined by Bouwknegt, Evslin and Mathai \cite{BEM}--induces suitable Courant algebroid isomorphisms, that we shall call \emph{dualities} in this paper.

We start introducing the notion of T-duality which we shall consider. Our Definition \ref{def:duality} below is a straightforward generalization of the main implication of 
\cite[Theorem 3.1]{CaGu}. Let $E$ be a Courant algebroid over a smooth manifold $M$, with anchor $\pi_E \colon E \to TM$. Assume that a torus $T$ acts on $M$ freely and properly, so that $M$ is a principal $T$-bundle over $M/T = B$, and that the action lifts to $E$ preserving the Courant algebroid structure. Then, we say that $(E,M,T)$ is an \emph{equivariant Courant algebroid} with base $B$. The \emph{simple reduction} of $E$ by $T$, which we denote
$$
E/T \to B,
$$
is a Courant algebroid over $B$, with anchor $\pi_{E/T} \colon E/T \to TB$, whose sheaf of sections is given by the invariant section of $E$, that is, $\Gamma(E/T) = \Gamma(E)^T$. 


\begin{definition}\label{def:duality}
Let $(E,M,T)$ and $(\hat E, \hat M, \hat T)$ be equivariant Courant algebroids over the same base $M/T = B = \hat M/\hat T$. We say that $(E,M,T)$ is dual to $(\hat E, \hat M, \hat T)$ if there exists an isomorphism of Courant algebroids between the simple reductions
$$
\psi \colon E/T \to \hat E/\hat T.
$$
In this situation, we will call $\psi$ the \emph{duality isomorphism}.
\end{definition}

The aim of this section is to understand the interplay between T-duality, in the sense of the previous definition, and the existence of solutions of the Ricci flow and the Killing spinor equations introduced in Section \ref{sec:flow}. For this, we need to introduce a notion of duality for pairs $(V_+,div)$, given by a generalized metric and a divergence operator. 

Given an equivariant Courant algebroid $(E,M,T)$ and a $T$-invariant generalized metric $V_+$ on $E$, we can push-forward $V_+$ along the bundle projection $p \colon M \to B$ to obtain a generalized metric with the same signature on the simple reduction 
$$
p_*V_+ \subset E/T.
$$
This way we obtain an identification between $T$-invariant generalized metrics on $E$ and generalized metrics on $E/T$. Similarly, given a $T$-invariant divergence operator $div \colon \Gamma(E) \to C^\infty(M)$ it induces a first-order differential operator 
$$
p_*div \colon \Gamma(E)^T = \Gamma(E/T) \to C^\infty(M)^T = C^\infty(B).
$$
The $\pi_E$-Leibniz rule for $div$ restricted to $p^*C^\infty(B)$ (see \eqref{eq:Leibnizdiv}), is precisely the $\pi_{E/T}$-Leibniz rule for $p_*div$, which thus defines a divergence operator on $E/T$.

\begin{definition}\label{def:dualpairs}
Let $(E,M,T)$ and $(\hat E, \hat M, \hat T)$ be dual equivariant Courant algebroids in the sense of Definition \ref{def:duality}. We will say that a $T$-invariant pair $(V_+,div)$ on $E$ is dual to a $\hat T$-invariant pair $(\hat V_+,\hat{div})$ on $\hat E$ if
$$
\psi (p_* V_+) = \hat p_* \hat V_+ \qquad \textrm{and} \qquad p_* div = \psi^* \hat p_* \hat{div},
$$
where $p \colon M \to B$ and $\hat p \colon \hat M \to B$ are the bundle projections.
\end{definition}

The next result shows that the Ricci tensor in Definition \ref{def:Ricci} is naturally exchanged under T-duality.

\begin{proposition}\label{prop:duality}
Let $(E,M,T)$ and $(\hat E, \hat M, \hat T)$ be equivariant dual Courant algebroids, endowed with dual invariant pairs $(V_+,div)$ on $E$ and $(\hat V_+,\hat{div})$ on $\hat E$, in the sense of Definition \ref{def:dualpairs}. Then
\begin{equation}\label{eq:Riccidual}
p_* Ric^\pm(V_+,div) = \psi^*\hat p_* Ric^\pm(\hat V_+,\hat{div}),
\end{equation}
where $Ric^\pm$ are as in Definition \ref{def:Ricci}.
\end{proposition}

The proof of Proposition \ref{prop:duality} requires to understand the relation between $T$-invariant generalized connections on $E$ and generalized connections on $E/T$. Denote by $\cD_E^T \subset \cD_E$ the space of $T$-invariant generalized connections on $E$ and by $\cD_{E/T}$ the space of generalized connections on the simple reduction $E/T$.

\begin{lemma}\label{lem:Gconnec}
There is a canonical identification $\cD_E^T = \cD_{E/T}$. In particular, $\cD_E^T$ is non-empty. 
\end{lemma}
\begin{proof}
The space $\cD_{E/T}$ is non-empty, since we can construct a generalized connection $\check D$ out of an orthogonal connection $\nabla^{E/T}$ on $E/T$ by 
$$
\check D_{e_1}e_2 = \nabla^{E/T}_{\pi_{E/T}(e_1)}e_2.
$$ 
The affine spaces $\cD_E^T$ and $\cD_{E/T}$ are both modelled on $\Gamma(E^* \otimes \mathfrak{o}(E))^T$, via the identification $\Gamma(E/T) = \Gamma(E)^T$, and therefore it is enough to prove that a generalized connection on $E/T$ induces an invariant connection on $E$. For this, we note that $E$ is locally generated by $\Gamma(E)^T$, and therefore given $\check D \in \cD_{E/T}$ we can define a generalized connection $D$ on $E$ by extending its action from $\Gamma(E)^T$ to $\Gamma(E)$ imposing the Leibniz rule with respect to the anchor $\pi_E$. The generalized connection $D$ is compatible with the pairing, as for $f \in C^\infty(M)$ and $e_1,e_2 \in \Gamma(E)^T$ we have
\begin{align*}
\pi_E(e_1)\la f e_2, e_3 \ra & = df(\pi_E(e_1)) \la e_2, e_3 \ra + f \pi_{E/T}(e_1)(\la e_2, e_3 \ra)\\
& = \la df(\pi_E(e_1)) + f \check D_{e_1}e_2,e_3\ra + f\la e_2,\check D_{e_1}e_3 \ra\\
& = \la D_{e_1}(fe_2),e_3\ra + \la e_2,D_{e_1}e_3 \ra.
\end{align*}
For the first equality we have used that $\la e_2, e_3 \ra$ is the pull-back of a function on $B$ and consequently
$$
\pi_E(e_1)\la e_2, e_3 \ra = \pi_{E/T}(e_1)\la e_2, e_3 \ra.
$$
\end{proof}

We are ready for the proof of Proposition \ref{prop:duality}.

\begin{proof}[Proof of Proposition \ref{prop:duality}]
Given a $T$-invariant divergence operator $div$ and a $T$-invariant torsion-free generalize connection $D$ on $E$ compatible with $V_+$ (see Proposition \ref{prop:existence}), we have
$$
div - div_D = \la e' , \cdot \ra \in \Gamma(E^*)^T.
$$
Considering $D' = D + \chi^{e^+}_+ + \chi^{e^-}_-$ with $\chi^{e^\pm}_\pm$ as in \eqref{eq:decompositionchipm} and setting 
$$
(e')^\pm = (r_\pm - 1)e^\pm,
$$
we obtain $div_{D'} = div$. Thus, the affine subspace $\cD(V_+,div)^T \subset \cD(V_+,div)$ of $T$-invariant generalized connections is non-empty. Furthermore, by Lemma \ref{lem:Gconnec} there is a canonical identification
$$
\cD(V_+,div)^T = \cD(p_* V_+,p_* div).
$$
Applying Lemma \ref{propo:Riccitorsion} and Proposition \ref{prop:weylfixed}, this implies
$$
p_* Ric^\pm(V_+,div) = Ric^\pm(p_* V_+,p_* div).
$$
The result follows now from
\begin{equation}\label{eq:Gconnecdual}
\{\psi_*D \; | \; D \in \cD(V_+,div)^T\} = \cD(\hat p_*  \hat V_+,\hat p_* \hat{div})
\end{equation}
which implies
$$
\psi^* Ric^\pm(\hat p_*\hat V_+,\hat p_* \hat{div}) = Ric^\pm(p_*  V_+,p_* div).
$$
\end{proof}

We state next our main result, which follows as a consequence of the proof of Proposition \ref{prop:duality} and equation \eqref{eq:Riccidual}.

\begin{theorem}\label{th:duality}
Let $(E,M,T)$ and $(\hat E, \hat M, \hat T)$ be dual equivariant Courant algebroids. Then:

\begin{enumerate}[i)]
\item If $(G_t,div_t)$ is an invariant solution of the Ricci flow \eqref{eq:Ricciflowpm} on $E$, and $(\hat G_t,\hat{div}_t)$ is a family of invariant dual pairs on $\hat E$, then $(\hat G_t,\hat{div}_t)$ is also a solution of the Ricci flow.

\item If $(G,div)$ is an invariant solution of the Killing spinor equations \eqref{eq:killing} on $E$, and $(\hat G,\hat{div})$ is an invariant dual pair on $\hat E$, then $(\hat G,\hat{div})$ is also a solution of the Killing spinor equations.
\end{enumerate}
\end{theorem}

\begin{proof}
The first part of the statement is a direct consequence of equation \eqref{eq:Riccidual}. As for the second part,   if $(V_+,div)$ is an invariant solution of the Killing spinor equations \eqref{eq:killing} then there exists an equivariant twisted spinor bundle $\mathcal{S}_+$ and a non-vanishing $\eta \in \Gamma(\mathcal{S}_+)^T$ satisfying \eqref{eq:killing}. The canonical operators $D^+_-$ and $\slashed D^+$ are constructed from an arbitrary choice of element $D \in \cD(V_+,div)$, that we can take to be $T$-invariant by Lemma \ref{lem:Gconnec} and the proof of Proposition \ref{prop:duality}, and thus these operators induce canonical operators $p_*D^+_-$ and $p_*\slashed D^+$ on $\Gamma(\mathcal{S}_+/T)$. Consequently, $p_* \eta \in \Gamma(\mathcal{S}_+/T)$ satisfy \eqref{eq:killing} and $(p_*  V_+,p_* div)$ is a solution of the Killing spinor equations on $E/T$. Finally, using the isomorphism $\psi$ in Definition \ref{def:duality}, we can regard $\mathcal{S}_+/T$ as a spinor bundle for $\hat p_* \hat V_+$ and $\eta$ as a non-vanishing section. Since the spaces of generalized connections in \eqref{eq:Gconnecdual} are exchanged under the isomorphism $\psi$, so are the operators $p_*D^+_-$ and $p_*\slashed D^+$. Therefore, $\cD(\hat p_*  \hat V_+,\hat p_* \hat{div})$ is also a solution of the Killing spinor equations. 
\end{proof}

A direct consequence of the previous theorem is that the stationary points of the flow, given by solutions of the Ricci-flat equation $Ric^+ = 0$, are exchanged under T-duality. A proof of this fact for transitive Courant algebroid obtained by reduction was provided in \cite{BarHek}, using the explicit form of the equations in this case (see \cite{GF} and Lemma \eqref{lem:Riccitr}). A proof of part i) of Theorem \ref{th:duality} for exact Courant algebroids was provided in \cite{Streets}, using an ad hoc definition of the Ricci tensor and explicit calculations.

We should mention that our proof also works when $T$ and $\hat T$ are substituted by arbitrary non-abelian groups, but we have not been able to find any interesting examples in this case. Based on our proof, we believe that Theorem \ref{th:Tduality} extends to a fairly general class of Poisson-Lie T-duals, in the sense of Klim\v c\'ik and \v Severa \cite{KlSevera}. We thank \v Severa for clarifications about the non-abelian setup.

It is interesting to notice that Proposition \ref{prop:specialholonomy} combined with Theorem \ref{th:duality} implies that special holonomy metrics with a continuous abelian group of isometries are preserved by T-duality, in the sense of Definition \ref{def:duality}. In this way, we have essentially recovered the observation of \cite{H2,Leung,LYZ,SYZ} that Calabi-Yau, $G_2$ and $Spin(7)$ metrics with torus symmetries are preserved by dualisation of the fibres. More interestingly, we believe that this is also true for the non-abelian Poisson-Lie T-duality \cite{KlSevera}. Even though compact manifolds with special holonomy have no continuous symmetries, it would be interesting to explore this perspective of the present work in the abundant local examples that exists in the literature (see e.g. \cite{Salamon}).

\subsection{Topological T-duality and the dilaton shift}\label{subsec:dilatonshift}

The aim of this section is to revisit some aspects of \emph{topological T-duality} for principal torus bundles, as defined in \cite{BEM} and further studied in \cite{CaGu}. 
The novelty here is to understand the notion of duality for pairs $(V_+,div)$ in Definition \ref{def:dualpairs} in this context, with the upshot of a mathematical explanation of the \emph{dilaton shift} in string theory, as originally observed by Buscher \cite{Buscher1} (see Remark \ref{rem:dilaton}). 

First, let us briefly summarise the basic definition following closely \cite{CaGu}. We are in the situation of the previous section with 
$(E,M,T)$, $(\hat E, \hat M,\hat T)$ a pair of equivariant Courant algebroids over a common base $B$. In this section we assume that $E$ and $\hat E$ are exact. We denote by 
$$
[H] \in H^3(M,\RR)^{T}, \qquad [\hat H] \in H^3(\hat M,\RR)^{\hat T}
$$ 
the \v Severa classes of $E$ and $\hat E$, respectively. Consider the fibre product $\overline M = M \times_B \hat M$ and the diagram
\begin{equation*}
  \xymatrix{
 & \ar[ld]_{\overline{p}} \overline M \ar[rd]^{\hat{\overline{p}}} & \\
 M \ar[rd]_{p} &  & \hat{M} \ar[ld]^{\hat{p}} \\
  & B & \\
  }
\end{equation*}
Following \cite{BEM}, we say that $(M,[H])$ is T-dual to $(\hat M,[\hat H])$ if there exists representants $H$ and $\hat H$ of the \v Severa classes such that
$$
\overline{p}^*H - \hat{\overline{p}}^* \hat H = d \overline{B},
$$
where $\overline{B} \in \Gamma(\Lambda^2 T^*\overline M)$ is a $T \times \hat T$-invariant $2$-form such that
$$
B \colon \Ker d \overline{p} \otimes \Ker d \hat{\overline{p}} \to \RR
$$
is non-degenerate. Typically $T$ and $\hat T$ are taken to be mutually dual tori, but this condition is irrelevant for our discussion. We now recall the salient implications of the T-duality relation \cite{BEM,CaGu}, in a way that is convenient for the present work.

\begin{theorem}[\cite{BEM,CaGu}]\label{th:Tduality}
If $(M,[H])$ and $(\hat M,[\hat H])$ are T-dual, then there exists an isomorphism of Courant algebroids between the simple reductions
$$
\psi \colon E/T \to \hat E/\hat T,
$$
and an isomorphism between the spaces of sections of the twisted spinor bundles
$$
\tau \colon \Gamma(\mathbb{S}) \to \Gamma(\hat{\mathbb{S}}),
$$
which exchanges the canonical Dirac generating operators in Theorem \ref{th:genoper} and which is compatible with Clifford multiplication, that is, 
$$
\tau(\slashed d_0 \alpha) = \hat{\slashed d_0}(\tau \alpha), \quad \textrm{and} \quad \tau(e \cdot \alpha) = \psi(e) \cdot \tau(\alpha)
$$ 
for all $e \in \Gamma(E)^{T}$ and $\alpha \in \Gamma(\mathbb{S})$.
\end{theorem}

Recall from Example \ref{example:dirac} that $(\Gamma(\mathbb{S}),\slashed d_0)$ corresponds to the differential complex $(\Gamma(\Lambda^*T^*M),d_H)$, with $d_H$ the twisted de Rham differential.
As a direct consequence of the previous theorem, topological T-duals provide a strong version of T-duals in the sense of Definition \ref{def:duality}. An interesting difference between Theorem \ref{th:Tduality} and Definition \ref{def:duality} is that $T$-duality--given by the isomorphism $\psi$--is a phenomenon which occurs on the common base $B$, while topological T-duality exchanges refined information on the total spaces of the fibrations via $\tau$. In fact, understanding the isomorphism $\tau$ in an invariant way from the point of view of the base is cumbersome, as the spinor bundle $\mathbb{S}$ involves a twist by the half-density bundle $|\det T^*M|^{\frac{1}{2}}$. This mismatch between the nature of the isomorphisms $\tau$ and $\psi$, is the responsible of a phenomenon known in the literature as \emph{dilaton shift}, that we try to explain now using Definition \ref{def:dualpairs}.

To provide a more invariant description, we shall work in the generality of Definition \ref{def:duality} and Definition \ref{def:dualpairs}, without using any explicit formula for the duality isomorphism.
Let $(V_+,div)$ and $(\hat V_+,\hat{div})$ be invariant dual pairs on exact equivariant Courant algebroids $E$ and $\hat E$, in the sense of Definition \ref{def:dualpairs}. We assume that the generalized metric $V_+$ satisfies $V_+ \cap T^*M = \{0\}$, so that it induces a standard $T$-invariant metric $g$ on $M$. Since $g$ is an invariant metric on $M$, it is equivalent to a triple $(\underline{g},\theta,h)$, where $\underline{g}$ is a metric on $B$, $\theta \colon TM \to VM = \Ker p$ is a connection on $M \to B$, and $h$ is a metric on the fibres of $M$, that is,
$$
g = h(\theta , \theta) + p^*\underline{g}.
$$
Using the explicit calculation in Section \ref{subsec:example}, we can express
\begin{equation}
div(e') = \mu_g^{-1} L_{\pi_E(e')}\mu_g - \la e,e'\ra
\end{equation}
for a suitable $e \in \Gamma(E)^T$, where $\mu_g$ denotes the pseudo-Riemannian density. 
We denote by $\nu \in \Gamma(|\det VM^*|)$ the density induced by $h$ on the vertical bundle $VM \subset TM$. Identifying 
$$
VM \cong M \times \mathfrak{t},
$$ 
where $\mathfrak{t}$ denotes the Lie algebra of $T$, we regard $\nu = |\det h|$ as an equivariant map $\nu \colon M \to |\det \mathfrak{t}^*|$ and, 
using that $T$ is unimodular, we obtain a well-defined map
$$
\nu \colon B \to |\det \mathfrak{t}^*|.
$$

\begin{definition}
The \emph{shifting $1$-form} 
of the $T$-invariant generalized metric $V_+$ is
$$
d \log \nu := \nu^{-1} d\nu \in \Gamma(T^*B).
$$
\end{definition}


Similarly, assuming that  $\hat V_+ \cap T^*\hat M = \{0\}$, 
we denote by $(\hat e, \hat g,\underline{\hat g},\hat \theta,\hat h, \hat \nu)$ the corresponding tuple associated to $(\hat V_+,\hat{div})$. 
We will assume that the induced metrics on the base coincide $\underline{\hat g} = \underline{g}$. For topological T-duality, as in Theorem \ref{th:Tduality}, this follows from \cite[Proposition 5.6]{Streets}.  

\begin{proposition}[Dilaton shift]\label{prop:dilatonshift}
With the notation above, assume that $\underline{\hat g} = \underline{g}$. Then, the equality $p_*div = \psi^*\hat p_* \hat{div}$ in Definition \ref{def:dualpairs} is equivalent to
\begin{equation}\label{eq:dilatonshift}
\hat e = \psi(e) + 2 d\log (\hat \nu/\nu),
\end{equation}
where $d\log (\hat \nu/\nu)$ is regarded as a section of $\hat E/\hat T$ using the anchor map. 
\end{proposition}
\begin{proof}
We denote $\underline{\pi} = \pi_{E/T}$. The statement follows from the formula
$$
div(e') = \mu_{\underline g}^{-1} L_{\underline{\pi}(e')}\mu_{\underline{g}} + 2\la \sigma_\nu,e'\ra - \la e,e'\ra,
$$
for any $e' \in \Gamma(E)^T$, since the metric on $B$ induced by $g$ and $\hat g$ coincide and $\psi$ is the identity along the Kernel of the anchor map $\underline{\pi}$. To prove this formula, we calculate
\begin{align*}
\mu_g^{-1} L_{\pi_E(e')}\mu_g =  \mu_{\underline g}^{-1} L_{\underline{\pi}(e')}\mu_{\underline{g}} + \nu^{-1} d\nu(\underline{\pi}(e')),
\end{align*} 
using the natural decomposition $\mu_g =  \mu_{\underline g} \otimes \nu$.
\end{proof}

Note that if $e = \varphi$ and $\hat e = \hat \varphi$ for invariant $1$-forms $\varphi$ on $M$ and $\hat \varphi$ on $\hat M$, respectively, then \eqref{eq:dilatonshift} is equivalent to
\begin{equation}\label{eq:dilatonexp}
\hat \varphi = \varphi - 2d \log \Bigg{(}\frac{|\det h|}{| \det \hat h|}\Bigg{)}.
\end{equation}
Going back to the setup of topological T-duality, using now the explicit \emph{Buscher rules} for the generalized metric we have (see \cite{CaGu})
$$
| \det \hat h| = |\det h|^{-1}.
$$
Thus, equation \eqref{eq:dilatonexp} tell us the precise way in which the `Weyl part' of the pure-type operators of $(V_+,div)$, given in \eqref{eq:LCvarphipure}, changes under the duality isomorphism. Finally, assuming that $\varphi$ and $\hat \varphi$ are exact, with the normalization $\varphi = 8d\phi$ and $\hat \varphi = 8d\hat\phi$ we obtain
$$
\hat \phi = \phi - \frac{1}{2}\log |\det h|.
$$
This equation is certainly not new: it was encountered by physicists in their computations of the dual Riemannian metric and dilaton field for T-dual sigma-models \cite{Buscher1} and carries the name of \emph{dilaton shift} (see Remark \ref{rem:killingphys}). We should stress that our formula \eqref{eq:dilatonshift} is independent of the explicit form of the duality isomorphism (provided that $\underline{\hat g} = \underline{g}$), and it is still valid if we replace $T$ and $\hat T$ in Definition \ref{def:duality} and Definition \ref{def:dualpairs} by arbitrary non-abelian unimodular Lie groups. Based on this, we believe that \eqref{eq:dilatonshift} generalizes to the non-abelian setup of Poisson-Lie T-duality \cite{KlSevera}. It would be interesting to compare our formula \eqref{eq:dilatonshift} with \cite[Eq. (3.16)]{OssaQue}, in the context of non-abelian duality in physics.

A different point of view on formula \eqref{eq:dilatonshift} is provided by the theory of Dirac generating operators, as considered in Section \ref{subsec:genoper}. Tracing back the construction of generating operators in the proof of Proposition \ref{prop:weylfixed} and using Lemma \ref{lem:Gconnec}, one can check that there is an identification 
between $T$-invariant generating operators on $\Gamma(\mathcal{S})$ of the Courant algebroid $E$ and generating operators on $\Gamma(\mathcal{S})^T = \Gamma(\mathcal{S}/T)$ of $E/T$. Using now the decomposition $\mu_g =  \mu_{\underline g} \otimes \nu$ of the Riemannian density, we obtain a commutative diagram of vector space isomorphisms
\begin{equation}\label{eq:diagram}
  \xymatrix{
   	\Gamma(\mathcal{S})^T  \ar[d]^{p_*} \ar[r]^{\otimes \mu_g^{\frac{1}{2}}}  & \Gamma(\mathbb{S})^T \ar[d]^{\otimes \nu^{-\frac{1}{2}}} & \\
    \Gamma(\mathcal{S}/T) \ar[r]^{\otimes \mu_{\underline g}^{\frac{1}{2}}} & \Gamma(\mathbb{S}_{E/T}) &
  }
\end{equation}
Using the horizontal arrows, we pull-back the canonical generating operator $\slashed d_0$ on $\Gamma(\mathbb{S})^T$ and $\underline{\slashed d}_0$ on $\Gamma(\mathbb{S}_{E/T})$ to the space of $T$-invariant sections $\Gamma(\mathcal{S})^T$ and to $\Gamma(\Gamma(\mathcal{S}/T) )$, respectively. Then, one has
$$
p_*(\slashed d_0) = \underline{\slashed d}_0 + \frac{1}{2} d \log \nu \cdot.
$$
Consequently, \eqref{eq:dilatonshift} can be interpreted as the discrepancy between $p_*(\slashed d_0)$ and the push-forward of the dual $p_*(\hat{\slashed d_0})$ using the isomorphism $\Gamma(\mathbb{S}_{E/T}) \cong \Gamma(\hat{\mathbb{S}}_{\hat E/\hat T})$ induced by $\tau$.

\section{The transitive case}\label{subsec:transitive}

In this section we apply our framework to the case of transitive Courant algebroids obtained by reduction. In particular, building on a result in \cite{grt}, we prove that the Hull-Strominger system of partial differential equations is invariant under T-duality (see Definition \ref{def:duality}).

Let $G$ be a Lie group, endowed with a biinvariant non-degenerate pairing on its Lie algebra, which we denote
$$
c \colon \mathfrak{g} \otimes \mathfrak{g} \to \RR.
$$
Let $p \colon P \to M$ be a principal $G$-bundle over a smooth manifold $M$ of dimension $n$. Provided that the (real) first Pontryagin class of $P$ with respect to $c$ vanishes
$$
p_1^c(P) = 0 \in H^4(M,\RR),
$$
there exists a transitive Courant algebroid $E$ over $M = P/G$ given by an extension
\begin{equation}\label{eq:transitiveext}
0 \to T^*M \to E \to TP/G \to 0.
\end{equation}
The anchor map of $E$ is given by the natural projection $TP/G \to TM$ and there is an isomorphism
$$
(\ad P,c) \cong \Ker \pi_E/(\Ker \pi_E)^\perp 
$$
as bundles of quadratic Lie algebras. The isomorphism classes of transitive Courant algebroids of this form are parametrized by the $H^3(M,\RR)$-torsor of \emph{real string classes} on $P$ \cite{BarHek}
$$
H^3_{str}(P,\RR) \subset H^3(P,\RR).
$$
Recall that a cohomology class $[\hat H] \in H^3(P,\RR)$ is said to be a string class if the restriction of $[\hat H]$ to
the fibres of $P$ coincides with the Cartan $3$-form class in $H^3(G,\RR)$ determined by the pairing $c$.

We fix a string class $[\hat H] \in H^3_{str}(P,\RR)$, and consider the associated transitive Courant algebroid $E$. Given a connection $\theta \in P$, there exists an isotropic splitting of $E$ and an isomorphism (see \cite[Proposition 2.4]{GF})
$$
E = T \oplus \ad P \oplus T^*,
$$
such that the pairing is given by
$$
\langle X + r + \xi,Y + t + \eta\rangle = \frac{1}{2}(\eta(X) +
\xi(Y)) + c(r,t),
$$
the anchor is the canonical projection $\pi_E(X + r + \xi) = X$ and the bracket is \cite{Severa}
\begin{equation}\label{eq:bracket}
  \begin{split}
    [X+r+\xi,Y+t+\eta]  = {} & [X,Y] + L_{X}\eta - \iota_{Y}d\xi + \iota_{Y}i_{X}H\\
    & - [r,t] - F(X,Y) + d^\theta_Xt - d^\theta_Y r\\
    & + 2c(d^\theta r,t) + 2c(\iota_XF,t) - 2c(\iota_Y F,r).
  \end{split}
\end{equation} 
In the last expression, $H$ is a $3$-form on $M$ (determined up to addition of exact $2$-forms), $F$ is the curvature of $\theta$, and $d^\theta$ is the induced covariant derivative on $\ad P$. The condition that \eqref{eq:bracket} satisfies the Jacobi identity  is equivalent to the \emph{Bianchi identity}
\begin{equation}\label{eq:bianchitrans}
dH = c(F \wedge F),
\end{equation}
which gives a distinguished trivialization of the first Pontryagin class of the principal bundle $p_1^c(P) = [c(F \wedge F)] \in H^4(M,\RR)$. We note that the pair $(H,\theta)$ determines a representant 
\begin{equation}\label{eq:CS}
\hat H = p^* H - CS(\theta) \in H^3_{str}(P,\RR),
\end{equation}
where $CS(\theta) = - \frac{1}{6}c(\theta \wedge [\theta,\theta]) + c(F \wedge \theta)$ is the Chern-Simons three-form of the connection $\theta$.

A generalized metric $V_+ \subset E$ is said to be \emph{admissible} if $r_+ = n$ and $V_-$ intersects $T^*$ transversally, that is, $V_- \cap T^* = \{0\}$ (notice the different sign notation in \cite[Definition 3.1]{GF}). An admissible generalized metric determines an isotropic splitting of $E$, a connection $\theta$ on $P$ and a three-form $H$ on $M$, so that the Courant structure is as above and we have the following simple expression for $V_\pm$
\begin{align*}
V_+ & = \{X + g(X): X \in \Gamma(T)\},\\
V_ - & = \{X - g(X) + r: X \in \Gamma(T), r \in \Gamma(\ad P)\},
\end{align*}
where $g$ is a (standard) metric on $M$. In the sequel we will assume that $g$ is Riemannian.

A torsion-free generalized connection for the admissible metric $V_+$ was constructed in \cite{grt} (see also \cite{GF}), and it is given explicitly by
\begin{equation}\label{eq:Levi-Civitaexptr}
  \begin{split}
    D^{0}_{a_-} c_- & {} = 2\Pi_-\(\nabla^{-1/3}_XZ - \frac{2}{3}g^{-1}c(i_XF,t) - \frac{1}{3}g^{-1}c(i_ZF,r)\)\\
    & \phantom{ {} = }+ d^\theta_X t - \frac{2}{3}F(X,Z) - \frac{1}{3}c^{-1}c(r,[t,\cdot]),\\
    D^{0}_{b_+} c_- &= 2\Pi_-\(\nabla^-_YZ - g^{-1}c(i_YF,t)\) + d^\theta_Y t - F(Y,Z),\\
    D^{0}_{a_-} b_+ &= 2\Pi_+\(\nabla^+_XY - g^{-1}c(i_YF,r)\),\\
    D^{0}_{b_+} d_+ &= 2\Pi_+\(\nabla^{1/3}_YW\).
  \end{split}
\end{equation}
where
\begin{equation}\label{eq:abcd}
  \begin{split}
    a_- &= X + r - gX,\\
    b_+ &= Y + gY,\\
    c_- &= Z + t - gZ,\\
    d_+ &= W + gW.
  \end{split}
\end{equation}
By \cite[Eq. (5.16)]{grt} it follows that $D^0$ differs from the generalized connection $D$ induced by
$$
\nabla^g \oplus d^\theta \oplus \nabla^{g^*}
$$
by a totally skew-symmetric element, and therefore $div_{D^0} = div_{D'}$. Now a direct calculation using \eqref{eq:divergenceex} shows that
$$
div_{D^0}(e) = div_{D}(e) = \mu_g^{-1}L_{\pi(e)}\mu_g.
$$
Denote $div = div_{D^0}$ and consider
\begin{equation}\label{eq:divergencevarphiex2}
div^\varphi = div - \la e, \cdot \ra,
\end{equation}
for $e \in \Gamma(E)$. Then, we can construct an element $D^\varphi \in \cD(V_+,div^\varphi)$ 
by
\begin{equation}\label{eq:LCvarphitr}
D^\varphi = D^0 + \frac{1}{n-1}\chi_+^{e^+} + \frac{1}{r_- - 1}\chi_-^{e^-}
\end{equation}

The following formulae for the Ricci tensors follows from a lengthy calculation using the curvatures of $D^\varphi$, similarly as in \cite[Proposition 4.2]{GF}.

\begin{lemma}\label{lem:Riccitr}
Let $e = \varphi^+ + \sigma^- + r$, for $\varphi, \sigma \in \Gamma(T^*)$ and $r \in \Gamma(\ad P)$. Then,
\begin{align*}
Ric^+(c^-,b^+) & = \(Ric^{\nabla^+} - F\circ F - \frac{1}{4} \nabla^+\varphi\)(Y,Z)\\
& - i_Y c\(d_A^*F - \frac{1}{2}*(F\wedge *H) + \frac{1}{4} F(g^{-1}\varphi,\cdot),t\) ,\\
Ric^-(b^+,c^-) & = \(Ric^{\nabla^-} - F\circ F - \frac{1}{4} \nabla^-\sigma\)(Y,Z)\\
& - i_Y c\(d_A^*F - \frac{1}{2}*(F\wedge *H) + \frac{1}{4} F(g^{-1}\sigma,\cdot),t\) - c([c^-,r],b^+).
\end{align*}
\end{lemma}

Using this formula, in the next result we give a sufficient condition for a pair $(V_+,div)$ to have the skew symmetry property.

\begin{proposition}\label{prop:Riccisymtr}
Let $E$ be a transitive Courant algebroid determined by $\hat H \in H^3_{str}(P,\RR)$. Assume that $E$ admits a spinor bundle $S$, and a root $(\det S)^{1/r_S}$. A pair $(V_+,div)$, given by an admissible generalized metric $V_+$ and a divergence operator $div$, 
has the skew-symmetry property in Definition \ref{def:Riccisym} provided that $div$ is induced by a Dirac generating operator.
\end{proposition}
\begin{proof}
By Theorem \ref{th:genoper}, the twisted spinor bundle $\mathbb{S}$ admits a canonical generating operator $\slashed d_0$. Using the (pseudo)Riemannian density on $M$ induced by $V_+$, we identify $\mathbb{S} \cong \mathcal{S}$. Then, applying again Theorem \ref{th:genoper}, any Dirac generating operator on $\mathcal{S}$ is of the form $\slashed d_0 + e \cdot$, where $e \in \Gamma(E)$ is such that $[\slashed d_0,e\cdot] \in C^\infty(M)$. Using property ii) in Definition \ref{def:genoper}, this implies that $[e,e'] = 0$ for any $e' \in \Gamma(E)$, and using the explicit formula for the bracket \eqref{eq:bracket}, we obtain $e = \varphi + r$, with $d \varphi = 0$, $d^\theta r = 0$ and $[r,\cdot] = 0$. The statement follows now from Lemma \ref{lem:Riccitr}.
\end{proof}

Next, we give an explicit formula for the Ricci flow on transitive Courant algebroids obtained by reduction.

\begin{proposition}\label{prop:Riccisymtr}
Let $E$ be a transitive Courant algebroid determined by $\hat H \in H^3_{str}(P,\RR)$. Let $(G_t,div_t)$ be a $1$-parameter family of pairs as before, with $G_t$ admissible for all $t$. Then, the Ricci flow \eqref{eq:Ricciflowpm} is equivalent to
\begin{equation}\label{eq:Ricciflowtr}
\begin{split}
\partial_t g & = -2\operatorname{Ric}^g + \frac{1}{2} H \circ H + 2F\circ F + \frac{1}{2} (\nabla^g\varphi)^{sym},\\
\partial_t b & = - d^*H - \frac{1}{4} \iota_{g^{-1}\varphi} H - \frac{1}{2}d\varphi,\\
\partial_t \theta & = - d_A^*F + \frac{1}{2}*(F\wedge *H) - \frac{1}{4} \iota_{g^{-1}\varphi}F.
\end{split}
\end{equation}
where $\nabla^g \varphi = (\nabla^g\varphi)^{sym} + d\varphi$.
\end{proposition}

\begin{proof}
Fixing an isotropic splitting for $E$, the family $(G_t,div_t)$ can be written as
$$
G_t = e^{(b_t,a_t)} \left( \begin{array}{ccc}
0 & 0 & g_t^{-1} \\
0 & - 1 & 0\\
g_t & 0 & 0 \end{array}\right) e^{(-b_t,-a_t)}
$$
$$
div_t(e) = \mu^{-1}_{g_t}L_{\pi(e)}\mu_{g_t} - \la e_t,e\ra
$$
where $e_t = \varphi_t^+ + \sigma_t^- + r_t$ as in Lemma \ref{lem:Riccitr}, and \cite{GF}
$$
e^{(b_t,a_t)} = \left( \begin{array}{ccc}
1 & 0 & 0 \\
a_t & 1 & 0\\
b_t - c(a_t \otimes a_t) & -2c(a_t,\cdot) & 1 \end{array}\right),
$$
for $b_t$ a $2$-form and $a_t$ an $\ad P$-valued $1$-form on $M$. Without loss of generality, we can assume $b_{t_0} = 0$ and $a_{t_0} = 0$, so that the derivative of $G_t$ at time $t_0$ is
$$
\partial_t G = [e^{(b_t,a_t)},G_{t_0}] + \left( \begin{array}{ccc}
0 & 0& - g_t^{-1} \partial_t g_t g_t^{-1} \\
0 & 0 & 0 \\
\partial_t g_t & 0 & 0 \end{array}\right).
$$
Notice that the connection on $P$ induced by $V_t$ is $\theta_t = \theta_{t_0} + a_t$ (see \cite[Proposition 3.4]{GF}). A direct calculation implies now that
$$
\langle \partial_t G  c_-,b_+ \rangle = \partial_t g (Z,Y) - \partial_t b (Z,Y) + 2c(\partial_t \theta(Y),t)
$$
at time $t_0$, and thus the result follows from Lemma \ref{lem:Riccitr}.
\end{proof}

We next state a result from \cite{grt}, which establishes the link between the Killing spinor equations \eqref{eq:killing} and the Hull-Strominger system \cite{HullTurin,Strom}. Assume that $M$ is a spin compact manifold of dimension $6$. Then, a solution of the Hull-Strominger system on $(M,P)$ is given by a tuple $(\omega,\Omega,\theta)$, given by a complex $3$-form $\Omega$ which determines an almost-complex structure on $M$, a hermitian form $\omega$, and a connection $\theta$ on $P$, satisfying equations
\begin{equation}\label{eq:Stromingersystem}
\begin{split}
d\Omega & = 0,\\
F_\theta \wedge \omega^2 & = 0,\\
d^*\omega - d^c \log \|\Omega\|_\omega & = 0,\\
dd^c \omega - c(F_\theta \wedge F_\theta) & = 0. 
\end{split}
\end{equation}
We refer to the review \cite{GF2} for a detailed discussion about these interesting equations. Note from the last equation that any solution of the Hull-Strominger system determines in particular a string class in $[\hat H] \in H^3_{str}(P,\RR)$, given by \eqref{eq:CS} with $H = d^c \omega$. We shall denote $E_{[\hat H]}$ the corresponding Courant algebroid.

\begin{theorem}[\cite{grt}]\label{th:Str}
Assume that $M$ is a spin compact manifold of dimension $6$. Solutions of the Hull-Strominger with string class $[\hat H] \in H^3_{str}(P,\RR)$ are in correspondence with solutions $(V_+,div)$ of the Killing spinor equations \eqref{eq:killing} on $E_{[\hat H]}$, such that $V_+$ is admissible and $\varphi = div - div_{D^0}$ is an exact $1$-form.
\end{theorem}

The particular instance of the equations \eqref{eq:Stromingersystem} which is relevant in physics is the case when the principal bundle $P$ is a fibred product
$$
P = P_M \times_M P_K,
$$
where $P_M$ is the bundle of oriented frames of $M$ and $P_K$ is principal bundle with compact structure group $K$, and the pairing $c$ is
$$
c = \tr_{\mathfrak{gl}} - \tr_{\mathfrak{k}}
$$
for $-\tr_{\mathfrak{gl}}$ and $-\tr_{\mathfrak{k}}$ positive definite forms on $\mathfrak{gl}(6,\RR)$ and $\mathfrak{k}$, respectively. In this situation, we obtain a strong form of Proposition \ref{prop:Ricciflat}, which recovers the `if part' of the main theorem in \cite{Ivan09}.

\begin{proposition}
Assume that $P = P_{K} \times_M P_{K'}$ and that $c = \tr_{\mathfrak{k}} - \tr_{\mathfrak{k}'}$, with $-\tr_{\mathfrak{k}}$ and $-\tr_{\mathfrak{k}'}$ positive definite forms. Then, any solution $(V_+,div)$ of the Killing spinor equations \eqref{eq:killing} with $V_+$ admissible is Ricci flat.
\end{proposition}
\begin{proof}
By assumption, an element $s \in \ad P$ decomposes as $s = r + t \in \ad P_K \oplus \ad P_{K'}$. The result follows applying Proposition \ref{prop:Ricciflat} and using the explicit formula for the Ricci tensor in Lemma \ref{lem:Riccitr}.
\end{proof}

We address now the main result of this section, which follows as a direct consequence of Theorem \ref{th:Str}, Theorem \ref{th:duality} and our formula for the \emph{dilaton shift} in Proposition \ref{prop:dilatonshift}.

\begin{theorem}\label{th:Strduality}
The solutions of the Hull-Strominger system are preserved by T-duality, as in Definition \ref{def:duality}.
\end{theorem}

\begin{proof}
By Theorem \ref{th:Str} any solution of the Hull-Strominger system with string class $[\hat H]$ determines a solution of the Killing spinor equations $(V_+,div^{\varphi})$ on a Courant algebroid $E_{[\hat H]}$, and with $\varphi$ exact. If $\hat E \to \hat M$ is a dual Courant algebroid (for suitable symmetry groups), then by Theorem \ref{th:duality} it also carries a solution of the Killing spinor equations $(\hat V_+,div^{\hat \varphi})$. Furthermore, by Proposition \ref{prop:dilatonshift}, the $1$-forms $\hat \varphi$ and $\varphi$ are related by the \emph{dilaton shift} equation \eqref{eq:dilatonexp}. Therefore, since $\varphi$ is exact it follows that $\hat \varphi$ is also exact, and by Theorem \ref{th:Str} it determines a solution of the Hull-Strominger system on the dual space. Note that \eqref{eq:dilatonshift} is valid for admissible metrics in the situation considered here.
\end{proof}

An explicit notion of T-duality for transitive Courant algebroids obtained by reduction and with integral string class has been proposed in \cite{BarHek}, under the name of \emph{heterotic T-duality}. By \cite[Proposition 4.13]{BarHek} it follows that our Theorem \ref{th:Strduality} applies in this situation. This has been recently used in \cite{GF2018} to find the first examples of T-dual solutions of the Hull-Strominger system on compact non-K\"ahler manifolds with different topology.

Theorem \ref{th:Strduality} 
builds towards the definition of a Strominger-Yau-Zaslow version of mirror symmetry for the Hull-Strominger system, as proposed in \cite{Yau2005}. This new incarnation of mirror symmetry should have very different features as the standard one, and we illustrate this with the following result. Recall that, in the standard SYZ picture, a Hermite-Yang-Mills connection $A$ on a Calabi-Yau manifold, regarded as a supersymmetric B-cycle, is mapped to a special Lagrangian on the mirror \cite{LYZ}. In contrast, a special holonomy metric equipped with an instanton has a completely different behaviour under heterotic T-duality.

To state the result, note that Theorem \ref{th:Str} admits generalizations for any dimension of $M$, and the resulting Hull-Strominger equations, written in terms of the corresponding $G$-structures, are those considered in \cite{FIVU,GMW,Ivan09}. Thus, the same proof leads to analogues of Theorem \ref{th:Strduality} for these equations, in a spin manifold $M$ of arbitrary dimension.

\begin{corollary}\label{cor:newmirror}
A metric $g$ with a parallel spinor, defining a $G$-structure and equipped with a $G$-instanton $A$, is mapped via heterotic T-duality to a solution of the Hull-Strominger system.
\end{corollary}

\begin{proof}
Consider $P = P_K \times_M P_K$, where $P_K$ is the principal bundle where the connection $A$ lives, and define $\theta = A \times A$. Then, $(g,A)$ defines a solution of the Killing spinor equations \eqref{eq:killing} (with $H = 0$ and $\varphi = 0$) in the transitive Courant algebroid determined by $[CS(\theta)] \in H^3_{str}(P,\RR)$. Here we use the invariant paring $c = \tr{\mathfrak{k}} - \tr{\mathfrak{k}}$ on $\mathfrak{k} \oplus \mathfrak{k}$, to define the Chern-Simons three-form $CS(\theta)$. The result now follows by application of Theorem \ref{th:duality}.
\end{proof}


\end{document}